\documentclass{article}
\usepackage[T1]{fontenc}
\usepackage[a4paper,left=2cm,right=2cm,top=3cm,bottom=3cm]{geometry}
\usepackage{csquotes}
\usepackage{amsmath,amsthm,amssymb}
\usepackage{algorithm}
\usepackage{algorithmicx}
\usepackage{algpseudocode}
\usepackage{nicefrac}
\usepackage{hyperref}
\usepackage[table, dvipsnames]{xcolor}

\definecolor{clUK}{RGB}{198,26,39}
\definecolor{clUKfb1}{RGB}{221,117,34}
\definecolor{clUKfb2}{RGB}{121,51,122}
\definecolor{clUKfb3}{RGB}{0,151,138}
\definecolor{clUKfb4}{RGB}{0,155,214}

\colorlet{clUK}{gray}
\colorlet{clUKfb1}{gray}
\colorlet{clUKfb2}{gray}
\colorlet{clUKfb3}{gray}
\colorlet{clUKfb4}{gray}

\usepackage{tikz, tikz-3dplot, amsmath}
\usetikzlibrary{calc}
\usetikzlibrary{decorations.pathreplacing,calligraphy,positioning}
\usetikzlibrary{pgfplots.groupplots}
\usepackage{pgfplots}
\pgfplotsset{compat=1.17}
\usepackage{pgfplotstable}
\usepackage{placeins}
\usepackage{orcidlink}

\usepackage
[
style=alphabetic,
maxnames=5, 
giveninits=true, 
isbn=false,url=false,eprint=false, doi=true,
backref=true
]
{biblatex}

\DeclareSourcemap{
  \maps[datatype=bibtex]{
    \map[overwrite]{
      \step[fieldsource=doi, final]
      \step[fieldset=url, null]
      \step[fieldset=eprint, null]
    }  
  }
}

\NewBibliographyString{toappearin}
\NewBibliographyString{submittedto}
\DefineBibliographyStrings{english}{%
  toappearin  = {to appear in},
  submittedto = {submitted to},
}

\renewbibmacro*{in:}{%
  \ifboolexpr{not test {\iffieldundef{pubstate}}
              and (test {\iffieldequalstr{pubstate}{toappearin}}
                   or test{\iffieldequalstr{pubstate}{submittedto}})}
    {\printtext{%
       \printfield{pubstate}\intitlepunct}%
     \clearfield{pubstate}}
    {\printtext{%
       \bibstring{in}\intitlepunct}}}

\usepackage{todonotes}

\theoremstyle{plain}
\newtheorem{theorem}{Theorem}[section]
\newtheorem{lemma}[theorem]{Lemma}

\theoremstyle{definition}
\newtheorem{definition}[theorem]{Definition}

\theoremstyle{remark}
\newtheorem{remark}[theorem]{Remark}

\DeclareMathOperator{\Span}{span}

\DeclareMathOperator{\POD}{POD}
\DeclareMathOperator{\eff}{eff}

\newcommand{\RR}{\mathbb{R}}

\newcommand{\RBestFineABS}{\eta_\star^\text{abs}}
\newcommand{\RBestFineREL}{\eta_\star^\text{rel}}

\newcommand{\RBestCcABS}{\eta_\mathfrak{c}^\text{abs}}
\newcommand{\RBestCcREL}{\eta_\mathfrak{c}^\text{rel}}

\allowdisplaybreaks

\begin{document}

\title{A reduced basis method for parabolic PDEs based on a space-time least squares formulation}

\author{
Michael Hinze%
\footnote{Mathematical Institute, University of Koblenz, Germany}
\orcidlink{0000-0001-9688-0150},\,%
Christian Kahle%
\footnotemark[1]
\orcidlink{0000-0002-3514-5512},\,%
Michael Stahl%
\footnotemark[1]
\orcidlink{0009-0000-4692-7328}
}

\date{\today}

\maketitle

\begin{abstract}
In this work, we present a POD-greedy reduced basis method for parabolic partial differential equations (PDEs), based on the least squares space-time formulation proposed in \cite{preprint_SP} that assumes only minimal regularity. We extend this approach to the parameter-dependent case. The corresponding variational formulation then is based on a parameter-dependent, symmetric, uniformly coercive, and continuous bilinear form. We apply the reduced basis method to this formulation, following the well-developed techniques for parameterized coercive problems, as seen e.g.~in reduced basis methods for parameterized elliptic PDEs. We present an offline–online decomposition and provide certification with absolute and relative error bounds. The performance of the method is demonstrated using selected numerical examples.
\end{abstract}

\section{Introduction.}
Let $V$ denote a real vector space and $a(\mu;\cdot,\cdot): V\times V \rightarrow \mathbb{R}$ a family of inner products such that $V(\mu):= (V,a(\mu;\cdot,\cdot))$ form Hilbert spaces. Here, $\mu \in\mathcal P$, where $\mathcal P$ denotes a parameter set. Furthermore, let $A(\mu) : V(\mu) \rightarrow V^\star(\mu)$ denote the Riesz isomorphism associated with $a(\mu;\cdot,\cdot)$, i.e.
\begin{align}\label{bilinearform}
    \langle A(\mu) \cdot, \cdot \rangle_{V^\star(\mu), V(\mu)} := a(\mu; \cdot, \cdot).
\end{align}
We develop a certified reduced basis method for parametrized parabolic equations of the form
\begin{equation}\label{APDE}
y_t+A(\mu)y=f(\mu) \text{ in } L^2(0,T;V^\star(\mu)), \quad y(0)= y_0(\mu) \text{ in } H,
\end{equation}
where $(V(\mu),H,V^\star(\mu))$ for every $\mu \in \mathcal P$ forms a Gelfand triple and $f(\mu) \in L^2(0,T; V^\star(\mu))$.

To derive a variational formulation for \eqref{APDE} we build upon the least squares space-time approach of \cite{preprint_SP}, which under natural regularity assumptions on the data leads to a variational formulation of the form
\begin{align}
    b(\mu; y, w) = l(\mu; w) \qquad \forall w \in W^\mu(0,T),
\end{align}
with continuous, symmetric and uniformly coercive bilinear forms $b(\mu;\cdot,\cdot)$, where $W^\mu(0,T) := \lbrace v \in L^2(0,T;V(\mu)) \mid v_t \in L^2(0,T;V(\mu)^\star)\rbrace$. For details we refer to Section \ref{sec:setting}. 
This establishes a space-time reduced basis framework based on the inner product induced by $b$, which inherits many advantages from the classical elliptic setting.
We focus on a POD-greedy method in the given natural space-time norm, demonstrating its performance with absolute and relative error bounds through two numerical examples.

\paragraph*{Novelty statement:}
\begin{itemize}
    \item[$\blacktriangleright$] Model order reduction for a least squares space-time formulation with natural $W(0,T)$ regularity of parameter-dependent parabolic PDEs (Section \ref{sec:setting} and \ref{sec:RB}),
    \item[$\blacktriangleright$] Reduced basis (absolute and relative) error estimators in a discrete $W(0,T)$ norm that allow an efficient offline-online decomposition (Section \ref{subsec:RBCert}),
    \item[$\blacktriangleright$] POD-greedy approach for the construction of reduced basis spaces in this context (Section \ref{sec:RB}).
\end{itemize}

\paragraph*{Literature.} The reduced basis method is well established for problems with a variational formulation involving a uniformly coercive and continuous bilinear form. This is the case, for example, in the setting of elliptic problems. For an overview of standard model order reduction techniques, see references \cites{Hesthaven2016}{haasdonk2017reduced}{GräßleHinzeVolkwein+2021+47+96}{Hesthaven_Pagliantini_Rozza_2022}{Hinze2023RB}{Quarteroni2016}. Of particular interest to us is the computational performance gain through \textit{offline-online decomposition} and the construction of best approximations by \textit{POD}.\\

Parabolic equations are often treated with time-stepping methods. For example, the backward Euler--Galerkin method is used to solve 3D problems in \cite{Gelsomino01082011}. In this method, the time-dependent partial differential equation (PDE) is solved by computing the solutions to a sequence of time-independent problems. In the parameter-dependent case, this allows the efficient application of reduced basis methods; however, the known error estimators need to work with the full time grid. Error estimators exist for the state at a given time point \cites[Prop.~2.80]{haasdonk2017reduced}[§~6.1.3]{Hesthaven2016}, as well as for some space-time energy norm \cites[Prop.~4.1]{greplpatera}[Prop.~2.82]{haasdonk2017reduced}. The latter is realized in \cite{greplpatera} by treating time as an additional parameter. Combining these error estimators with a POD-greedy approach establishes an exponential decay of the Kolmogorov N-width, and hence convergence of the reduced basis method. In \cite[Prop.~4.4]{haasdonkCONVrb} this is achieved by applying POD to the \textit{flat set}, which is the union of all local (time-independent) steps of the high-fidelity solutions. In \cite{greplpatera}, the authors also use a greedy algorithm and demonstrate its performance on a heat shield example. \\

In recent years, there has been a growing interest in the use of space-time formulations for reduced basis methods. In \cite{rovas2006reduced}, reduced-basis output bound methods are extended from the elliptic case to parabolic partial differential equations by treating time as an additional parameter, in a manner similar to \cite{greplpatera}. In this study, the authors employ a space-time formulation and solve the high-fidelity problems numerically using the discontinuous Galerkin method in time. As a numerical example, they consider a \textit{thermal fin} problem.

In \cite{eftang2011hp}, the reduced basis approach is employed within an \textit{hp}  framework for linear and nonlinear parabolic problems. The parameter domain is split into subdomains for which individual RB spaces are constructed. The POD method is used in time and a greedy algorithm is used with respect to the parameters. The high-fidelity problems are solved using a backward Euler and Crank--Nicolson scheme. Error estimators are also provided in the aforementioned spatio-temporal energy norm.

Fully space-time reduced basis methods for parabolic problems are introduced in \cites{urban2012new}{urban2014improved}. In these, the variational formulation contains a bilinear form which is non-symmetric in time and satisfies an inf-sup condition. Linear systems are obtained using a Petrov--Galerkin approximation. The reduced basis space is constructed with respect to the spatial solutions, providing an error estimator in a discrete norm arising from the Riesz lift. This approach is applied to option pricing in \cite{Mayerhofer2016}.

A similar variational formulation in space-time is employed for the numerical treatment of time-periodic problems in \cite{STEIH2012710}. There, the authors use basis functions that are periodic in time, realized through wavelets. They also provide \textit{a posteriori} error estimators for the reduced basis approach.

For an overview of reduced basis methods, particularly in the context of space-time methods for parabolic problems, see \cite[§~1.4.2]{Hinze2023RB}.
In this work, we also employ the concept of a discrete $W(0, T)$ norm from \cite{urban2014improved}, with the aim of introducing a reduced basis approach for the full $W(0, T)$ using a symmetric space-time formulation.\\

There are also contributions investigating certified reduced basis methods for nonlinear, time-dependent equations. In \cite{yano2014space}, a certified reduced basis method for the Boussinesq system is proposed. Quasilinear parabolic equations are investigated in \cite{HK21} and \cite{hoppe2024posteriori}. Nonsmooth semilinear parabolic equations are considered in \cite{bernreuther2024adaptive}. A general framework for nonlinear parabolic equations with empirical interpolation is presented in \cite{benaceur2018progressive}.\\

The convergence of reduced basis methods is measured by the Kolmogorov N-width, which quantifies the worst best-approximation that appears in a given parameter set. 
In \cite[Thm.~3.1]{ohlberger}, a general proof of exponential convergence is provided in the case where the variational formulation involves a bilinear form that is uniformly coercive, continuous, and allows a parameter decomposition. We also refer to \cite{unger2019kolmogorov} in the context of model reduction for \textit{LTI systems} and the recent work \cite{arbes2025kolmogorov} investigating the Kolmogorov $N$-width for linear transport problems.\\

In \cite{preprint_SP}, we introduced a least squares space-time formulation for parabolic problems with natural regularity. Our reduced basis approach is based on this. The bilinear form in the variational formulation is symmetric, uniformly coercive and continuous.
For numerical treatment, the problem is reformulated as an equivalent saddle-point equation, and it has been shown that the Galerkin approximation converges to the continuous solution.

\paragraph*{Outline.} In Section \ref{sec:setting}, we describe the problem setting in the parameter-dependent case and present the least squares formulation, as well as an equivalent parameterized saddle point problem. In Section \ref{sec:FE}, we introduce a space-time discretization using tensorial space-time finite elements. This makes it possible to state the high-fidelity problem. The reduced problem is formulated in Section \ref{sec:RB}, where we also introduce the concept of space-time POD. In Section \ref{subsec:RBCert}, we derive absolute and relative error estimators and demonstrate their performance using two numerical examples in Section \ref{subsec:numericsRB}.

\section{Problem setting.}\label{sec:setting}
Let $(V,(\cdot,\cdot)_{V})$ and $(H,(\cdot,\cdot)_H)$ 
denote separable Hilbert spaces with the properties 
$V \hookrightarrow H\equiv H^\star \hookrightarrow V^\star$, so that $(V,H,V^\star)$ forms a Gelfand triple. Let $\mathcal{P}$ denote some parameter space and $\mu \in \mathcal{P}$ some arbitrary, but fixed parameter. We denote with $a(\mu; \cdot, \cdot) : V \times V \rightarrow \RR$ a parameter-dependent, symmetric, uniformly continuous and coercive bilinear form, defining an inner product on $V$. With this we introduce the Hilbert space $V(\mu)$ as the space $V$ equipped with the inner product $a(\mu; \cdot, \cdot)$,
\begin{align}\label{eq:RBbilinA}
    (u, v)_{V(\mu)} := a(\mu; u, v) \qquad \forall u, v \in V .
\end{align}
By these assumptions, $(V(\mu), (\cdot, \cdot)_{V(\mu)})$ is a family of parameter-dependent separable Hilbert spaces with the property
\begin{align}
    V(\mu) \hookrightarrow H \equiv H^\star \hookrightarrow V(\mu)^\star \qquad \forall \mu \in \mathcal{P},
\end{align}
so that $(V(\mu),H,V(\mu)^\star)$ is a family of parameter-dependent Gelfand triples.\\

We choose some reference parameter $\overline{\mu} \in \mathcal{P}$ to identify $V = V(\overline{\mu})$ and from here onwards assume that all inner products $a(\mu; \cdot, \cdot)$ are equivalent on $V$ in the way that $a(\mu; \cdot, \cdot)$ is uniformly coercive and continuous, i.e., there exist positive constants $c_s, c_c$ independent of $\mu \in \mathcal P$ such that
\begin{align} \label{eq:rbcoercivitycontconst}
\vert a(\mu;u,v)\vert \le c_s \|u\|_V\|v\|_V \quad \text{and} \quad a(\mu;u,u) \ge c_c \|u\|_V^2 \qquad \forall u,v\in V.    
\end{align}
Then the norms on $V$ and $V(\mu)$ as well as on $V^{\star}$ and $V(\mu)^{\star}$ are equivalent with the estimates
\begin{equation}\label{eq:abschaetzungVundVst}
 \sqrt{c_c}\|v\|_V \le \|v\|_{V(\mu)} \le \sqrt{c_s}\|v\|_V \qquad \text{and} \qquad \frac{1}{\sqrt{c_s}}\|v\|_{V^{\star}} \le \|v\|_{V(\mu)^{\star}} \le \frac{1}{\sqrt{c_c}}\|v\|_{V^{\star}}.
\end{equation}
\\
For given $T > 0$ and for any fixed $\mu \in \mathcal{P}$ we define the parabolic spaces
\begin{equation}
    W^{\mu}(0, T) := \{v\in L^2(0,T;V(\mu)), v_t\in L^2(0,T;V(\mu)^\star)\} \quad \text{and} \quad W(0,T) := W^{\overline{\mu}}(0, T).
\end{equation}
The inner product is given by 
\begin{equation}
    (u,v)_{W^{\mu}(0,T)} := \int_0^T (u_t,v_t)_{V(\mu)^\star} + \int_0^T(u,v)_{V(\mu)}
\end{equation}
and the norm induced by this inner product is denoted by $\|v\|_{W^{\mu}(0,T)} := (v,v)_{W^{\mu}(0,T)}^{1/2}$.
Using \eqref{eq:abschaetzungVundVst} we directly obtain the equivalence of $(\cdot, \cdot)_{W^{\mu}(0,T)}$ and $(\cdot, \cdot)_{W(0,T)}$ on $W(0,T)$ and consequently the equivalence of the induced norms,
\begin{equation} \label{eq:normsWequi}
    \sqrt{\min\left\lbrace c_c,\frac{1}{c_s}\right\rbrace }\|v\|_{W(0,T)} \le \|v\|_{W^{\mu}(0,T)} \le \sqrt{\max\left\lbrace c_s,\frac{1}{c_c} \right\rbrace }\|v\|_{W(0,T)} \quad \forall v\in W(0,T).
\end{equation}
Thus, considered as sets, the spaces $W^{\mu}(0,T)$ and $W(0,T)$ are equal.\\ 
With the given family of inner products we associate a family of parameter-dependent operators $A(\mu) : V \rightarrow V^{\star}$ via 
\begin{equation}
    \left\langle A(\mu)u, v\right\rangle_{V^{\star}, V} :=  a(\mu; u, v) \qquad \forall \mu \in \mathcal{P} \quad \forall u,v \in V .
\end{equation}
With this, $A(\mu) : V(\mu) \rightarrow V(\mu)^\star$ is the Riesz isomorphism. From here onwards we call $R(\mu) := A(\mu)^{-1} : V(\mu)^\star \rightarrow V(\mu)$ the Riesz lift, satisfying for any $\phi \in V^\star$
\begin{align}
    (R(\mu) \phi,v)_{V(\mu)} =  \langle \phi,v\rangle_{V^\star,V}  \qquad \forall  v \in V
\end{align}
and for $\phi,\psi \in V^\star$ we set
\begin{align}
    (\phi,\psi)_{V(\mu)^\star}:= (R(\mu) \phi, R(\mu) \psi)_{V}. 
\end{align}
By definition of $(\cdot,\cdot)_{V(\mu)}$ it follows that $A(\mu) R(\mu) \phi = \phi$ for all $\phi\in V^\star$ and 
$R(\mu) A(\mu) v = v$ for all $v \in V $, see \cite{Wloka1987,Troeltzsch2010}.\\

In this setting we naturally extend $R(\mu)$ from $V(\mu)^\star$ to $L^2(0,T;V(\mu)^\star)$ and understand $R(\mu)$ as operator $R(\mu) :L^2(0,T;V(\mu)^\star) \to L^2(0,T;V(\mu))$ and define $R(\mu)\phi$ for any $\phi \in L^2(0,T;V^\star)$ as $R(\mu)\phi(t)$ for almost every $t$. With this, $R(\mu)\phi$ fulfills
\begin{equation}
    \int_0^T (R(\mu)\phi, v)_{V(\mu)} = \int_0^T \langle \phi, v \rangle_{V^\star, V} \qquad \forall v \in V .
\end{equation}

Finally, assume that parameter-dependent $y_0(\mu) \in H$ and $f(\mu) \in L^2(0,T; V^\star)$ are given, where $f$ is assumed to be uniformly bounded  with respect to $\mu$ as well, i.e.~there exists $c_s^f > 0$ such that
\begin{align}\label{eq:cont_of_f}
    \vert\vert f(\mu) \vert\vert_{L^2(0,T;V^\star)} \leq c_s^f.
\end{align}

For fixed $\mu \in \mathcal{P}$ we solve the parameter-dependent parabolic problem
\begin{equation}\label{eq:parab_mu}\tag{$\dagger$}
    y_t + A(\mu)y = f(\mu) \text{ in } L^2(0,T;V^\star), \quad y(0)=y_0(\mu) \text{ in } H.
\end{equation}
Since $\mu \in \mathcal{P}$ is arbitrary but fixed, \eqref{eq:parab_mu} admits a unique variational solution $y(\mu) \in W(0,T)$, see e.g.~\cite{Wloka1987,Thomee2006}. In the following we omit writing the dependence on $\mu$ for $y$ whenever it is clear.

\begin{remark}
The standard setting we have in mind is the heat equation with homogeneous Dirichlet boundary data and parameter-dependent diffusion, source term and initial condition. In this application for an open and bounded domain $\Omega \subset \mathbb{R}^n$ we have $V= H^1_0(\Omega)$, $H=L^2(\Omega)$, $V^\star=H^{-1}(\Omega)$ and
\begin{equation}
    \langle A(\mu) y,v\rangle_{V^\star,V} = \int_\Omega \kappa(\mu) \nabla y \nabla v ,
\end{equation}
where $\kappa(\mu) \in L^\infty(\Omega)$ is a parameter-dependent diffusivity term.
\end{remark}

\subsection{A least squares space-time approach.}
For the variational formulation of \eqref{eq:parab_mu} we use the least-squares approach  introduced in  \cite{preprint_SP} and  extend it to the parameter-dependent case. For this, we consider the quadratic  minimization problem 
\begin{equation}\label{eq:minresid}
    \min_{v\in W(0,T)} \|r(\mu; v)\|^2_{{ L^2(0,T;V(\mu)^\star)}\times H} 
    := \|v_t+A(\mu)v-f(\mu)\|^2_{ L^2(0,T;V(\mu)^\star)}+\|v(0)-y_0(\mu)\|^2_H.
\end{equation}
Following \cite{preprint_SP} the necessary and sufficient first order optimality condition gives rise to a variational formulation, defined below.
\begin{definition}[Continuous problem]
We consider the following formulation to solve \eqref{eq:parab_mu}:\\
For fixed $\mu \in \mathcal{P}$, find $y \in W(0,T)$ such that
\begin{align}
    b(\mu; y,w) = l(\mu; w) \quad \forall w \in W(0,T),\label{prob:P_mu}\tag{$\ddagger$}
\end{align}
where for arbitrary $v,w\in W(0,T)$ the forms $b(\mu; \cdot, \cdot)$ and $l(\mu; \cdot)$ are defined as
\begin{align}
  b(\mu; v,w) &:= \int_0^T (v_t,w_t)_{V(\mu)^\star}+ \int_0^T(v,w)_{V(\mu)}  + (v(T),w(T))_H , \label{eq:b_def_mu}\\
    l(\mu; w)&:=(y_0(\mu),w(0))_H 
    + \int_0^T (f(\mu),w_t)_{V(\mu)^\star}
    + \int_0^T \langle f(\mu),w\rangle_{V^\star,V}. 
    \label{eq:l_def_mu}
\end{align}
\end{definition}
With the results in \cite{preprint_SP} together with the equivalency of norms \eqref{eq:normsWequi}, the uniform coercivity and continuity of $b(\mu; \cdot, \cdot)$ directly follows. Uniform continuity of $l(\mu; \cdot)$ follows with \eqref{eq:cont_of_f}. The well-posedness of formulation \eqref{prob:P_mu} for any fixed $\mu \in \mathcal P$ then is obtained with Lax--Milgram's theorem.

\subsection{Reformulation as a saddle point problem.}
For the numerical treatment it is convenient to reformulate \eqref{prob:P_mu} as a saddle point problem. For fixed $\mu \in \mathcal{P}$ we introduce $\hat{a}(\mu) : W(0,T) \times W(0,T) \to \RR$, $\hat{b} : W(0,T) \times L^2(0,T; V) \to \RR$, $\hat{c}(\mu) : L^2(0,T; V) \times L^2(0,T; V) \to \RR$, $\hat{l}_1(\mu) : W(0,T) \to \RR$ and $\hat{l}_2(\mu) : L^2(0,T; V) \to \RR$ by
\begin{align}
    \hat{a}(\mu; y,w) &= (y(T),w(T))_H  + \int_0^T(y,w)_{V(\mu)} ,\\
    \hat{b}(w,q) &=  \int_0^T\langle w_t,q\rangle_{V^\star,V} , \\
    \hat{c}(\mu; p,q) &= \int_0^T(p,q)_{V(\mu)} ,\\
    \hat{l}_1(\mu; w) &= (y_0(\mu),w(0))_H  + \int_0^T\langle f(\mu), w \rangle_{V^\star,V} ,\\
    \hat{l}_2(\mu; q) &= \int_0^T \langle f(\mu), q \rangle_{V^\star,V}.
\end{align}
For given $\mu \in \mathcal P$ the problem to solve then is given by finding $(y,p) \in W(0,T)\times L^2(0,T;V)$ such that for all $(w,q) \in W(0,T)\times L^2(0,T;V)$ there holds
\begin{equation}\label{eq:contSPorig}\tag{\ensuremath{\text P^\mu}}
\begin{aligned}
    \hat{a}(\mu; y,w) + \hat{b}(w,p) &= \hat{l}_1(\mu; w),\\
    \hat{b}(y,q) -\hat{c}(\mu; p,q) &= \hat{l}_2(\mu; q).
\end{aligned}    
\end{equation}
The equivalence of \eqref{eq:contSPorig} and \eqref{prob:P_mu} is shown in \cite[Lemma 2.6]{preprint_SP}. In particular, $y$ is the solution to \eqref{prob:P_mu} if $(y,p)$ solves \eqref{eq:contSPorig}.

\section{The high fidelity problem.}\label{sec:FE}
To obtain a high fidelity solution of the saddle point problem \eqref{eq:contSPorig} we introduce a Galerkin approximation. Let $I := (0,T]$. We define the finite dimensional spaces
\begin{align}
J_P &= \Span\{\psi_p\mid p=1,\ldots,P\} \subset L^2(I), \label{JP} \\
K_M &= \Span\{\chi_m\mid m=1,\ldots, M, (\chi_m)_t \in J_P\} \subset H^1(I), \label{KM} \\
V_N &= \Span\{\phi_n\mid n=1,\ldots,N\} \subset V. \label{VN}
\end{align}
With them we define the space-time spaces
\begin{align} 
    Q_d = J_P \otimes V_N &=
    \left\lbrace q_d(t,x)=\sum_{p=1}^P\sum_{n=1}^N q_n^p\psi_p(t)\phi_n(x), \;\; q_n^p \in \RR \; \forall n,p \right\rbrace \subset L^2(0,T;V), \label{eq:Qd}\\
    W_d = K_M \otimes V_N &= 
    \left\lbrace w_d(t,x)=\sum_{m=1}^M\sum_{n=1}^N w_n^m\chi_m(t)\phi_n(x), \;\; w_n^m \in \RR \; \forall n,m \right\rbrace \subset W(0,T), \label{eq:Wd}
\end{align}
that we use to approximate the solution $(y,p)$ to \eqref{eq:contSPorig}. For fixed $\mu \in \mathcal{P}$ the discrete version of the saddle point problem \eqref{eq:contSPorig} then is given by
seeking $y_d \in W_d$ and $p_d \in Q_d$ such that
\begin{equation}\tag{\ensuremath{\text P^\mu_\ensuremath{d}}}\label{eq:hfproblemSP}
\begin{aligned}
    \hat{a}(\mu; y_d,w_d) + \hat{b}(w_d,p_d) &= \hat{l}_1(\mu; w_d) \qquad &&\forall w_d \in W_d ,\\
    \hat{b}(y_d,q_d) -\hat{c}(\mu; p_d,q_d) &= \hat{l}_2(\mu; q_d) \qquad &&\forall q_d \in Q_d.
\end{aligned}    
\end{equation}
To assemble system \eqref{eq:hfproblemSP} we define $T_t, M_t \in \RR^{M \times M}$, $M_t^\psi \in \RR^{P \times P}$, $Z_t \in \RR^{P \times M}$ and
$A_x(\mu),M_x \in \RR^{N\times N}$ by
\begin{align}
    (T_t)_{i,j}&= \chi_j(T)\chi_i(T),
    & (M_t)_{i,j} &= \int_0^T \chi_j\chi_i, 
    & (Z_t)_{i,j}&= \int_0^T (\chi_j)_t \psi_i, \\
    (M_t^\psi)_{i,j} &= \int_0^T \psi_j\psi_i, 
    & (M_x)_{i,j} &= (\phi_j,\phi_i)_H,
    &(A_x(\mu))_{i,j} &= \langle A(\mu)\phi_j,\phi_i \rangle_{V^\star, V}.
\end{align}
In addition, we introduce $R_0^t \in \RR^{M}$, $R_0^x(\mu) \in \RR^{N}$, $F_1(\mu) \in \RR^{NM}$ and $F_2(\mu) \in \RR^{NP}$ according to
\begin{align}
    (R_0^t)_m &:= \chi_m(0),      
    &(F_1(\mu))_{(m-1)N+n} &:= \int_0^T \langle f(\mu),\chi_m\phi_n\rangle_{V^\star,V},\\
    (R_0^x(\mu))_n &:= (y_0(\mu),\phi_n)_H, 
    & (F_2(\mu))_{(m-1)N+n} &:= \int_0^T \langle f(\mu),\psi_m\phi_n\rangle_{V^\star,V}.
\end{align}
Finally, to compactify the notation let $S_d(\mu) \in \RR^{N(M + P) \times N(M + P)}$ and $s_d(\mu) \in \RR^{N(M + P)}$ defined by
\begin{equation}
    S_d(\mu) := \begin{pmatrix} T_t \otimes M_x + M_t \otimes A_x(\mu) & Z_t^T \otimes M_x \\
    Z_t \otimes M_x & -M_t^\psi \otimes A_x(\mu)
    \end{pmatrix},
    \qquad
    s_d(\mu) := \begin{pmatrix} R_0^t \otimes R_0^x(\mu) + F_1(\mu) \\ F_2(\mu)
    \end{pmatrix} .
\end{equation}
Here, $\otimes$ denotes the Kronecker product of two matrices. We refer to \cite[§~1.3.6]{Golub} for more details and stress that in $s_d(\mu)$ we also understand the vectors as matrices concerning $\otimes$. Problem \eqref{eq:hfproblemSP} then is realized by solving
\begin{equation}\label{eq:SP_RB_diskr}
    S_d(\mu) \begin{pmatrix} \Vec{y}\\ \Vec{p}\end{pmatrix} = s_d(\mu),
\end{equation}
where $\Vec{y} \in \RR^{NM}$ and $\Vec{p} \in \RR^{NP}$ denote the coefficient vectors of $y_d \in W_d$ and $p_d \in Q_d$. 
\begin{remark}\label{rem:stiffnessTime}
The convergence of the finite element approximation to the continuous solution is shown in \cite{preprint_SP}. The standard setting are piecewise linear and globally continuous finite elements (CG 1) for spanning $K_M$ and piecewise constant elements (DG 0) for $J_P$. 
\end{remark}

\paragraph*{Inner product and norm of $W_d$.} To avoid an explicit calculation of the exact Riesz representer of every basis function for the use in the $V^\star$ inner product we introduce the discrete Riesz lift. With this we reduce \eqref{eq:hfproblemSP} to an equation for $y_d$ only, from which we define a natural norm on $W_d$, which is also used to state error estimators in Section \ref{subsec:RBCert}.\\

Let $R_d(\mu) : L^2(0,T; V^\star) \rightarrow Q_d$ be an approximation of the (time-extended) Riesz lift $R(\mu) : L^2(0,T; V^\star) \rightarrow L^2(0,T;V)$, given by
\begin{align}\label{eq:rieszapproximations}
    \int_0^T (R_d(\mu)\phi,q_d)_{V(\mu)}
 :=  \int_0^T \langle \phi,q_d\rangle_{V^\star,V}
  =   \int_0^T (R(\mu) \phi,q_d)_{V(\mu)}
  \quad \forall \phi \in L^2(0,T;V^\star), \, q_d \in Q_d.  
\end{align}
Using that $(w_d)_t \in Q_d$ for all $w_d \in W_d$, \eqref{eq:hfproblemSP} is equivalent to
\begin{align}
    b_d(\mu; y_d(\mu),w_d) = l_d(\mu; w_d) \quad \forall w_d \in W_d,
\end{align}
where for arbitrary $v_d, w_d \in W_d$ the forms $b_d(\mu; \cdot, \cdot)$ and $l_d(\mu; \cdot)$ are defined as
\begin{align}\label{eq:def_bd_mu}
  b_d(\mu; v_d,w_d) &:= \int_0^T (R_d(\mu)(v_d)_t, R_d(\mu)(w_d)_t)_{V(\mu)} + \int_0^T (v_d, w_d)_{V(\mu)} + (v_d(T), w_d(T))_H ,\\
    l_d(\mu; w_d)&:=(y_0,w_d(0))_H 
    + \int_0^T \langle f(\mu),R_d(\mu)(w_d)_t\rangle_{V^\star,V}
    + \int_0^T \langle f(\mu),w_d\rangle_{V^\star,V}. 
\end{align}
\begin{definition}[Energy inner product and norm]\label{defi:normequiDISCR}
On $W_d$ we introduce an inner product
\begin{align}
    (v_d, w_d)_\mu := b_d(\mu; v_d, w_d) \qquad \forall v_d, w_d \in W_d
\end{align}
as well as the induced norm $\vert \vert w_d \vert \vert_\mu := (w_d, w_d)_\mu^{1/2}$.\\For the given $\overline{\mu} \in \mathcal{P}$ we equip $W_d$ with the inner product $(\cdot, \cdot)_{W_d} := b_d(\overline{\mu}, \cdot, \cdot)$ and set $\vert \vert \cdot \vert \vert_{W_d} := (\cdot, \cdot)_{W_d}^{1/2}$.
\end{definition}
With the above definitions and by \eqref{eq:rbcoercivitycontconst} we directly obtain that
\begin{align}\label{eq:normsEnergyEqui}
    \min\left\lbrace c_c,\frac{1}{c_s}\right\rbrace \vert \vert w_d \vert \vert_{W_d}^2 \leq b_d(\mu; w_d, w_d) = \vert \vert w_d \vert \vert_\mu^2 \leq \max\left\lbrace c_s,\frac{1}{c_c} \right\rbrace \vert \vert w_d \vert \vert_{W_d}^2 ,
\end{align}
where $c_c$ and $c_s$ are upper resp.~lower bounds for the coercivity and continuity constants of $a(\mu;\cdot,\cdot)$ from \eqref{eq:rbcoercivitycontconst}. Hence, $b_d$ is uniformly coercive and the norms from Definition \ref{defi:normequiDISCR} are equivalent on $W_d$. \\

For later use in Section \ref{subsec:RBCert}, we also define discrete coercivity and continuity constants. Let 
\begin{equation}
    \mathfrak{c}_c(\mu) := \underset{v \in V_N}{\inf} \frac{a(\mu; v,v)}{\vert \vert v \vert \vert^2_V} \geq c_c \quad \text{and} \quad \mathfrak{c}_s(\mu) := \underset{u \in V_N}{\sup} \underset{v \in V_N}{\sup} \frac{a(\mu; u,v)}{\vert \vert u \vert \vert_V \vert \vert v \vert \vert_V} \leq c_s
\end{equation}
denote the discrete parameter-dependent coercivity and continuity constants of $a(\mu; \cdot, \cdot)$. Furthermore we set
\begin{equation}
    \alpha(\mu) := \min\left\lbrace \mathfrak{c}_c(\mu),\frac{1}{\mathfrak{c}_s(\mu)}\right\rbrace \quad \text{and} \quad  \alpha_{\text{LB}} := \min\left\lbrace c_c,\frac{1}{c_s}\right\rbrace \leq \alpha(\mu) ,
\end{equation}
where the latter is a lower bound for the coercivity constant of $b_d$. Then property \eqref{eq:normsEnergyEqui} generalizes to
\begin{align}\label{eq:normsEnergyEqui2}
    \alpha_{\text{LB}} \vert \vert w_d \vert \vert_{W_d}^2 \leq \alpha(\mu) \vert \vert w_d \vert \vert_{W_d}^2 \leq \vert \vert w_d \vert \vert_\mu^2 .
\end{align}

\section{Reduced basis method.}\label{sec:RB}
The idea of \textit{model order reduction} consists in replacing $W_d$ and $Q_d$ with some low-dimensional, problem-specific subspaces $W_{\texttt{L}} := \Span \left\lbrace \xi_{l} \, \vert \, l = 1, ..., \texttt{L}\right\rbrace \subset W_d$ and $Q_{\texttt{K}} := \Span \left\lbrace \rho_{k} \, \vert \, k = 1, ..., \texttt{K}\right\rbrace \subset Q_d$, where $(\xi_{l})_{l=1}^{\texttt{L}}$, resp.~$(\rho_{k})_{k=1}^{\texttt{K}}$ denote the \textit{reduced bases}. These spaces are in general constructed out of previously computed solutions or expert knowledge. For a moment, we assume that the reduced basis spaces are given and discuss the construction of such spaces at the end of this section in Remark \ref{rem:choiceQK} and Algorithm \ref{greedy}.

\begin{definition}[Reduced problem] 
For fixed $\mu \in \mathcal{P}$ we seek for $y_\texttt{rb} \in W_\texttt{L}$ and $p_\texttt{rb} \in Q_\texttt{K}$ such that
\begin{equation}\tag{\ensuremath{\text P^\mu_\texttt{rb}}}\label{eq:reduziertesproblem}
\begin{aligned}
    \hat{a}(\mu; y_\texttt{rb},w_\texttt{rb}) + \hat{b}(w_\texttt{rb},p_\texttt{rb}) &= \hat{l}_1(\mu; w_\texttt{rb}) \qquad &&\forall w_\texttt{rb} \in W_\texttt{L},\\
    \hat{b}(y_\texttt{rb},q_\texttt{rb}) -\hat{c}(\mu; p_\texttt{rb},q_\texttt{rb}) &= \hat{l}_2(\mu; q_\texttt{rb}) \qquad &&\forall q_\texttt{rb} \in Q_\texttt{K}.
\end{aligned}    
\end{equation}
\end{definition}
We highlight that in fact $y_\texttt{rb} = y_\texttt{rb}(\mu)$, however, we omit writing the dependence on $\mu$ for readability whenever it is clear.

For some $y_\texttt{rb} \in W_\texttt{L}$ and $p_\texttt{rb} \in Q_\texttt{K}$ we denote with $\Vec{u}_y \in \RR^{\texttt{L}}$ resp.~$\Vec{u}_p \in \RR^{\texttt{K}}$ the corresponding coefficient vectors in $W_\texttt{L}$ resp.~$Q_\texttt{K}$, i.e.
\begin{align}\label{eq:redsolyp}
    y_\texttt{rb} = \sum_{l=1}^{\texttt{L}} (\Vec{u}_y)_l \xi_l \in W_\texttt{L} \quad \text{and} \quad p_\texttt{rb} = \sum_{k=1}^{\texttt{K}} (\Vec{u}_p)_k \rho_k \in Q_\texttt{K} .
\end{align}
By construction $\xi_l$ and $\rho_k$ can be represented by linear combinations of the bases of $W_d$ and $Q_d$ and consequently $y_\texttt{rb}$ and $p_\texttt{rb}$ can be written as elements of $W_d$ and $Q_d$. For obtaining the coefficients of $y_\texttt{rb}$ and $p_\texttt{rb}$ with respect to the bases of $W_d$ and $Q_d$ let $B_W \in \mathbb{R}^{NM \times \texttt{L}}$, $B_Q \in \mathbb{R}^{NP \times \texttt{K}}$ s.t.~the $l$-th column of $B_W$ is the coefficient vector of $\xi_{l}$ and the $k$-th column of $B_Q$ the coefficient vector of $\rho_{k}$ with respect to the high fidelity bases. Thus, with $w_i$ and $q_i$ being the basis functions of $W_d$ and $Q_d$ respectively, it holds that
\begin{align}
    \xi_l = \sum_{i=1}^{NM}(B_W)_{il}w_i \quad \text{and} \quad \rho_k = \sum_{i=1}^{NP}(B_Q)_{ik}q_i .
\end{align}
Then, with $\Vec{y}_\texttt{rb} \in \RR^{NM}$ and $\Vec{p}_\texttt{rb} \in \RR^{NP}$ being the coefficient vectors of $y_\texttt{rb}$ and $p_\texttt{rb}$ in $W_d$ resp.~$Q_d$, it holds that
\begin{equation}\label{eq:reconstrHFspace}
    \Vec{y}_\texttt{rb} = B_W \Vec{u}_y \qquad \text{and} \qquad \Vec{p}_\texttt{rb} = B_Q \Vec{u}_p .
\end{equation}
Relating these equations with \eqref{eq:hfproblemSP} and \eqref{eq:reduziertesproblem} we obtain the representation \begin{equation}\label{eq:linproblemreduced}
    S_\texttt{rb}(\mu) \begin{pmatrix} \Vec{u}_y\\ \Vec{u}_p\end{pmatrix} = s_\texttt{rb}(\mu)
\end{equation}
of \eqref{eq:reduziertesproblem}, where $S_\texttt{rb}(\mu) \in \RR^{(\texttt{L}+\texttt{K}) \times (\texttt{L}+\texttt{K})}$ 
and $s_\texttt{rb}(\mu) \in \RR^{\texttt{L}+\texttt{K}}$ are given as
\begin{align}
\begin{split}
    S_\texttt{rb}(\mu) &:= \begin{pmatrix}B_W & \\ & B_Q\end{pmatrix}^T S_d(\mu) \begin{pmatrix}B_W & \\ & B_Q\end{pmatrix} \\&\hphantom{:}= \begin{pmatrix} B_W^T (T_t \otimes M_x + M_t \otimes A_x(\mu)) B_W & B_W^T (Z_t^T \otimes M_x) B_Q \\
    B_Q^T (Z_t \otimes M_x) B_W & -B_Q^T(M_t^\psi \otimes A_x(\mu))B_Q
    \end{pmatrix},\end{split}
    \\
    s_\texttt{rb}(\mu) &:= \begin{pmatrix}B_W & \\ & B_Q\end{pmatrix}^T s_d(\mu) = \begin{pmatrix} B_W^T (R_0^t \otimes R_0^x(\mu) + F_1(\mu)) \\ B_Q^T F_2(\mu)
    \end{pmatrix} .
\end{align}

\paragraph*{Efficiency through offline-online decomposition.} The previous formulation can be exploited if the appearing matrices and vectors are parameter-separable. Then $S_\texttt{rb}(\mu)$ and $s_\texttt{rb}(\mu)$ from \eqref{eq:linproblemreduced} can be written as a sum of parameter-independent parts multiplied with parameter-dependent scalar coefficients and are accessible without assembling $S_d(\mu)$ and $s_d(\mu)$. In an \textit{offline phase}, the parameter-independent parts are computed. Then, in an \textit{online phase}, these are used to build $S_\texttt{rb}(\mu)$ and $s_\texttt{rb}(\mu)$, where no assembling is necessary. This approach is known as the \textit{offline-online decomposition} in literature. The offline phase is computationally expensive, but only needs to be performed once. The online phase, in contrast, requires significantly less computational power. We apply this approach in the following.\\

For some $Q_S, Q_s > 0$ and functions $(\theta_S^q)_{q=1}^{Q_S}, (\theta_s^q)_{q=1}^{Q_s} : \mathcal{P} \rightarrow \mathbb{R}$, specified later, we derive a representation of the high fidelity operators $S_d(\mu)$ and $s_d(\mu)$ of the form
\begin{equation}
    S_d(\mu) = \sum_{q=1}^{Q_S} \theta_S^q(\mu) S_q \qquad \text{and} \qquad s_d(\mu) = \sum_{q=1}^{Q_s} \theta_s^q(\mu) s_q .
\end{equation}
In the \textit{offline phase} the operators 
\begin{equation}
    S_\texttt{rb}^q := \begin{pmatrix}B_W & \\ & B_Q\end{pmatrix}^T S_q \begin{pmatrix}B_W & \\ & B_Q\end{pmatrix} \qquad \text{and} \qquad s_\texttt{rb}^q := \begin{pmatrix}B_W & \\ & B_Q\end{pmatrix}^T s_q
\end{equation}
are built once. In the \textit{online phase} we build 
\begin{equation}
    S_\texttt{rb}(\mu) = \sum_{q=1}^{Q_S} \theta^q_S(\mu) S_\texttt{rb}^q \qquad \text{and} \qquad s_\texttt{rb}(\mu) = \sum_{q=1}^{Q_s} \theta^q_s(\mu) s_\texttt{rb}^q
\end{equation}
for arbitrary $\mu \in \mathcal{P}$ and solve \eqref{eq:reduziertesproblem}. The solutions are then represented in the high fidelity space by \eqref{eq:reconstrHFspace}. 

For this we from now on assume that $A(\mu)$, $f(\mu)$ and $y_0(\mu)$ are parameter-separable meaning that there exist $Q_A, Q_f, Q_y >0$ and bounded parameter functions $(\theta^{q}_{A})_{q=1}^{Q_A}, (\theta^{q}_{f})_{q=1}^{Q_f}, (\theta^{q}_{y_{0}})_{q=1}^{Q_y} : \mathcal{P} \rightarrow \mathbb{R}$, $\theta^{q}_{A} > 0$ and parameter-independent operators $A_{q} : V \rightarrow V^\star$, $q=1,...,Q_A$, $f_{q} \in L^2(0,T; V^\star)$, $q=1,...,Q_f$, and $y_{0,q} \in H$, $q=1,...,Q_y$, such that 
\begin{equation}
A(\mu) = \sum_{q=1}^{Q_{A}} \theta^{q}_{A}(\mu) A_{q}, \qquad f(\mu) = \sum_{q=1}^{Q_{f}} \theta^{q}_{f}(\mu) f_{q}, \qquad y_0(\mu) = \sum_{q=1}^{Q_{y}} \theta^{q}_{y_{0}}(\mu) y_{0,q} .
\end{equation}
We introduce the discretization of the parameter-independent parts of the operators by $A_x^q \in \RR^{N\times N}$, $R_{0}^{x,q} \in \RR^N$, $F_{1}^q \in \RR^{NM}$ and $F_{2}^q \in \RR^{NP}$, where
\begin{alignat}{5}
\left(A_{x}^{q}\right)_{i,j} &:= \left\langle A_{q}\phi_{j}, \phi_{i} \right\rangle_{V^{\star}, V}, \qquad  &i,j &= 1, ..., N, \, &q &= 1, ..., Q_{A}, \\
\left(R_{0}^{x,q}\right)_{n} &:= (y_{0,q},\phi_n)_H  &n &= 1, ..., N, \, &q &= 1, ..., Q_{y}, \\
\left(F_{1}^q\right)_{(m-1)N+n} &:= 
\int_{0}^{T} \left\langle f_{q}, \chi_m\phi_{n} \right\rangle_{V^{\star}, V},\quad  &m &= 1, ..., M, \; &n &= 1, ..., N, \, q = 1, ..., Q_{f}, \\
\left(F_{2}^{q}\right)_{(m-1)N+n} &:= 
\int_{0}^{T} \left\langle f_{q}, \psi_m\phi_{n} \right\rangle_{V^{\star}, V},\quad &m &= 1, ..., P, \, &n &= 1, ..., N, \, q = 1, ..., Q_{f}.
\end{alignat}
Now, $S_d(\mu)$ and $s_d(\mu)$ are parameter-separable as desired with $Q_S := Q_A + 1$ and $Q_s := Q_{y} + Q_{f}$ and
\begin{alignat}{5}
    &\theta_S^q(\mu) := \theta_A^q(\mu), \quad &S_q &:= \begin{pmatrix} M_t \otimes A_x^q & 0 \\
    0 & -M_t^\psi \otimes A_x^q
    \end{pmatrix}, \quad & q &= 1, ..., Q_A , \\
    &\theta_S^q(\mu) := 1, \quad &S_q &:= \begin{pmatrix} T_t \otimes M_x & Z_t^T \otimes M_x \\
    Z_t \otimes M_x & 0
    \end{pmatrix}, \quad & q &= Q_A + 1 , \\
    &\theta_s^q(\mu) := \theta^{q}_{y_{0}}(\mu), \quad &s_q &:= \begin{pmatrix} R_0^t \otimes R_{0}^{x,q} \\ 0 \end{pmatrix}, \quad & q &= 1, ..., Q_y , \\
    &\theta_s^q(\mu) := \theta^{i}_{f}(\mu), \quad &s_q &:= \begin{pmatrix} F_{1}^i \\ F_{2}^i
    \end{pmatrix}, \quad & q &= i+ Q_y, \; i = 1, ..., Q_f .
\end{alignat}

\paragraph*{Inner product for functions from $W_\texttt{L}$.} We are also interested in calculating the $W_d$ inner product without having access to the matrix representation of $b_d$ for the use in Section \ref{subsec:RBCert}. The computation of it can be simplified, if $B_Q$ is chosen in dependence of $B_W$ as we show in the following lemma.
\begin{lemma}\label{lem:choiceRB}
Let $v_\texttt{rb}, w_\texttt{rb} \in W_\texttt{L}$ with coefficient vectors $\Vec{v}, \Vec{w} \in \RR^\texttt{L}$. If $B_Q = (M_t^\psi \otimes A_x(\overline{\mu}))^{-1} (Z_t \otimes M_x) B_W$ then the inner product $(v_\texttt{rb}, w_\texttt{rb})_{W_d}$ is given by
\begin{align}
    (v_\texttt{rb}, w_\texttt{rb})_{W_d} = \Vec{v}^T B_W^T(T_t \otimes M_x + M_t \otimes A_x(\overline{\mu})) B_W \Vec{w} + \Vec{v}^T B_Q^T (M_t^\psi \otimes A_x(\overline{\mu})) B_Q \Vec{w} .
\end{align}
\end{lemma}
\begin{proof}
This follows by a direct calculation, since
\begin{align}
    (v_\texttt{rb}, w_\texttt{rb})_{W_d} &= b_d(\overline{\mu}, v_\texttt{rb}, w_\texttt{rb}) = \Vec{v}^T B_W^T (T_t \otimes M_x + M_t \otimes A_x(\overline{\mu}) + Z_t^T (M_t^\psi)^{-1} Z_t \otimes M_x A_x^{-1}(\overline{\mu}) M_x) B_W\Vec{w} \\
    &= \Vec{v}^T \left( B_W^T(T_t \otimes M_x + M_t \otimes A_x(\overline{\mu})) B_W + B_Q^T (M_t^\psi \otimes A_x(\overline{\mu})) B_Q \right) \Vec{w} .
\end{align}
\end{proof}
\begin{remark}\label{rem:choiceQK}
    We require the choice $B_Q = (M_t^\psi \otimes A_x(\overline{\mu}))^{-1} (Z_t \otimes M_x) B_W$ from now on for the rest of this paper. \\
    For the application of the POD method, described in Remark \ref{rem:POD}, we need efficient access to the matrix representation of the $W_d$ inner product in the reduced basis space. Therefore, for every system solve of \eqref{eq:hfproblemSP}, leading to some $\Vec{y} \in \RR^{NM}$, which is the coefficient vector of the solution $y_d$, we also compute 
\begin{align}
    (M_t^\psi \otimes A_x(\overline{\mu}))^{-1} (Z_t \otimes M_x) \Vec{y} .
\end{align}
\end{remark}
\begin{remark}[Formulation of space-time POD]\label{rem:POD} In practice, reduced basis spaces are often further reduced in dimension by removing unnecessary \textit{information}. This can done by the POD method. Let $\mathcal{P}_h \subseteq \mathcal{P}$ with $\texttt{L} = \vert \mathcal{P}_h \vert$. Upon $\texttt{L}$ computed solutions of this parameter set, the POD method can be used to produce a smaller reduced basis space $W_{\texttt{POD}} \subseteq  \Span{\lbrace y_d(\mu) \, \vert \, \mu \in \mathcal{P}_h\rbrace}$ that minimizes the best-approximation error between $W_\texttt{POD}$ and $W_d$,
\begin{align}
    \sqrt{\frac{1}{\texttt{L}} \sum_{\mu \in \mathcal{P}_h}\underset{w \in W_{\texttt{POD}}}{\inf} \vert \vert y_d(\mu)-w\vert\vert_{W_d}}.
\end{align}
Together with Lemma \ref{lem:choiceRB} the implementation of POD can be realized analogously to the case of reduced basis methods for elliptic problems, see e.g.~\cite[p.~33]{Hesthaven2016} for an implementation. We also highlight that with this formulation all advantages from the classical POD approach take over to the space-time $W_d$ norm. We refer to \cites[§~3.2.1]{Hesthaven2016}[§~2.4.5]{haasdonk2017reduced}{GräßleHinzeVolkwein+2021+47+96} for more details. This approach will be used together with a \textit{greedy algorithm} to recover the classical \textit{POD-greedy approach} \cites[Def.~2.92]{haasdonk2017reduced}[§~6.1.2]{Hesthaven2016} in the following.
\end{remark}

In Algorithm \ref{greedy} we state a procedure to generate a reduced basis for a given set of parameters. Here, $\eta$ denotes an error estimator, specified later.

\begin{algorithm}
\caption{POD-greedy procedure}\label{greedy}
\begin{algorithmic}[1]
\State Choose $\mathcal{S}_{\text{train}} \subset \mathcal{P}$, an arbitrary $\mu^{1} \in \mathcal{P}$, $\mu^1 \not\in \mathcal{S}_{\text{train}}$, $\epsilon_{\text{tol}} > 0$, $\texttt{L}_1, \texttt{L}_2 \in \mathbb{N} \setminus \lbrace 0 \rbrace$ and an error estimator $\eta$
\State Set $\texttt{L} := 1$, $W_{1} := \Span \lbrace y_{d}(\mu^{1}) / \vert\vert y_{d}(\mu^{1})\vert\vert_{W_d} \rbrace$, $\mathcal{Z} := \lbrace y_{d}(\mu^{1}) \rbrace$
\While{$\underset{\mu \in \mathcal{S}_{\text{train}}}{\max} \eta(\mu) > \epsilon_{\text{tol}}$}
\State $\texttt{L} \gets \texttt{L} + \texttt{L}_1$
\State $l:=1$
\While{$l \leq \texttt{L}_2$}
\State $\mu^{\texttt{L}}_{l} := \underset{\mu \in \mathcal{S}_{\text{train}}}{\arg \max} \; \eta(\mu)$\label{alg:online_phase}
\State $\mathcal{S}_{\text{train}} := \mathcal{S}_{\text{train}}\setminus \lbrace \mu^{\texttt{L}}_{l} \rbrace$
\State $l \gets l+1$
\EndWhile
\State $\mathcal{Z} = \mathcal{Z} \cup \lbrace y_{d}(\mu^{\texttt{L}}_1), ..., y_{d}(\mu^{\texttt{L}}_{\texttt{L}_2}) \rbrace$\label{alg:sys_solves}
\State $W_{\texttt{L}} := \POD_{\texttt{L}}(\mathcal{Z})$\label{alg:POD}
\EndWhile
\end{algorithmic}
\end{algorithm}
In each step of Algorithm \ref{greedy} the reduced basis space $W_\texttt{L}$ in enriched by $\texttt{L}_1$ basis functions. For that, either some absolute or relative error estimator is used. We introduce such estimators in Section \ref{subsec:RBCert}. In each iteration we select $\texttt{L}_2$ parameters that maximize the error estimator. We compute the high fidelity solutions for those parameters and store all of them in some set $\mathcal{Z}$. The reduced basis space $W_\texttt{L}$ then is obtained by applying POD onto $\mathcal{Z}$ and taking the first $\texttt{L}$ basis functions generated by POD.

\section{Certification.}\label{subsec:RBCert}
We enrich our reduced basis space by those high fidelity solutions, which are very likely to maximize the best-approximation error of the reduced basis space $W_{\texttt{L}}$. Also, we quantify this error for the use in the Algorithm \ref{greedy}. Therefore, we in the following introduce absolute and relative error estimators in the $W_d$ norm similarly to the elliptic-PDE case. For that, we first transfer a well-known estimator for elliptic PDEs into the space-time setting in Theorem \ref{thm:RBerror_ex}. From that we derive (absolute and relative) error estimators, which allow a faster computation by some offline-online decomposition.\\

\noindent To start with, we introduce the discrete residual as
\begin{equation}
    r_d(\mu, \cdot) := l_d(\mu; \cdot) - b_d(\mu; y_\texttt{rb}(\mu), \cdot) \in (W_d)^\star,
\end{equation}
with $l_d$ and $b_d$ as introduced in \eqref{eq:def_bd_mu}. Let $y_d(\mu)$ be the solution of \eqref{eq:hfproblemSP} for some fixed $\mu \in \mathcal{P}$ and $y_\texttt{rb}(\mu)$ the solution to \eqref{eq:reduziertesproblem}. We denote the absolute and relative errors of the reduced basis approximation with 
\begin{align}
    \epsilon^\text{abs}(\mu) := \vert \vert y_{d}(\mu) - y_{\texttt{rb}}(\mu) \vert \vert _{W_d} \quad \text{and} \quad \epsilon^\text{rel}(\mu) := \frac{\vert \vert y_{d}(\mu) - y_{\texttt{rb}}(\mu) \vert \vert _{W_d}}{\vert \vert y_{d}(\mu) \vert \vert _{W_d}} .
\end{align}
Analogously to the elliptic-PDE case, we estimate these errors by the norm of the residual.
\begin{theorem}\label{thm:RBerror_ex}
Let $\Tilde{r}_d(\mu) \in W_d$ be the Riesz representation of $r_d(\mu, \cdot)$. In the given reduced basis setting it holds that
\begin{alignat}{3} \label{eq:RB_first_abs_est}
    &\epsilon^\text{abs}(\mu) &&\leq \frac{\vert \vert \Tilde{r}_d(\mu) \vert \vert _{W_d}}{\alpha(\mu)} &&=: \RBestFineABS(\mu), \\
    &\epsilon^\text{rel}(\mu) &&\leq \frac{2\vert \vert \Tilde{r}_d(\mu) \vert \vert _{W_d}}{\alpha(\mu)\vert \vert y_{\texttt{rb}}(\mu) \vert \vert _{W_d}} &&=: \RBestFineREL(\mu), \quad \text{if} \; \; \RBestFineREL(\mu) \leq 1 . \label{eq:RB_first_rel_est}
\end{alignat}
\end{theorem}
\begin{proof}
We proceed similarly to \cites[Prop.~4.4]{Hesthaven2016}[Prop.~2.24, 2.27]{haasdonk2017reduced}. Let $\mu \in \mathcal{P}$ be arbitrary but fixed, $y_d(\mu) \in W_d$ the solution to \eqref{eq:hfproblemSP} and $y_{\texttt{rb}}(\mu) \in W_\texttt{L}$ the solution to \eqref{eq:reduziertesproblem}. We denote the error $e_d(\mu) \in W_d$ by
\begin{align}
    e_d(\mu) := y_d(\mu) - y_\texttt{rb}(\mu).
\end{align}
Then for arbitrary $w_d \in W_d$ it holds that
\begin{align}
    b_d(\mu; e_d(\mu), w_d) &= b_d(\mu; y_d(\mu), w_d) - b_d(\mu; y_{\texttt{rb}}(\mu), w_d) \\
    &= l_d(\mu; w_d) - b_d(\mu; y_{\texttt{rb}}(\mu), w_d) \\
    &= r_d(\mu; w_d) \\
    & = (\Tilde{r}_d(\mu), w_d)_{W_d} .
\end{align}
With this we receive together with the Cauchy–Schwarz inequality and property \eqref{eq:normsEnergyEqui2} that
\begin{align}
    \vert \vert y_{d}(\mu) - y_{\texttt{rb}}(\mu) \vert \vert_\mu^2 &= \vert \vert e_d(\mu) \vert \vert_\mu^2 = b_d(\mu; e_d(\mu), e_d(\mu)) \\
    &= (\Tilde{r}_d(\mu), e_d(\mu))_{W_d} \\
    &\leq \vert \vert \Tilde{r}_d(\mu) \vert \vert_{W_d} \vert \vert e_d(\mu) \vert \vert_{W_d} \\
    &\leq \vert \vert \Tilde{r}_d(\mu) \vert \vert_{W_d} \frac{\vert \vert e_d(\mu) \vert \vert_\mu}{\sqrt{\alpha(\mu)}}.
\end{align}
Thus, $\vert \vert y_{d}(\mu) - y_{\texttt{rb}}(\mu) \vert \vert_\mu \leq \frac{\vert \vert \Tilde{r}_d(\mu) \vert \vert_{W_d}}{\sqrt{\alpha(\mu)}}$. With this and using property \eqref{eq:normsEnergyEqui2} again we obtain \eqref{eq:RB_first_abs_est} since
\begin{align}
    \vert \vert y_{d}(\mu) - y_{\texttt{rb}}(\mu) \vert \vert_{W_d}^2 &= \vert \vert e_d(\mu) \vert \vert_{W_d}^2 \leq \frac{\vert \vert e_d(\mu) \vert \vert_{\mu}^2}{\alpha(\mu)} \leq \frac{\vert \vert \Tilde{r}_d(\mu) \vert \vert_{W_d}^2}{\alpha(\mu)^2}.
\end{align}
We now use the notation in \eqref{eq:RB_first_abs_est} and \eqref{eq:RB_first_rel_est} and assume that $\RBestFineREL(\mu) \leq 1$. Then it holds that
\begin{align}
    \vert \vert y_d(\mu) \vert \vert_{W_d} &= \vert \vert y_{\texttt{rb}}(\mu) \vert \vert_{W_d} + \vert \vert y_d(\mu) \vert \vert_{W_d} - \vert \vert y_{\texttt{rb}}(\mu) \vert \vert_{W_d} \\
    &\geq \vert \vert y_{\texttt{rb}}(\mu) \vert \vert_{W_d} + \vert \vert y_d(\mu) - y_{\texttt{rb}}(\mu) \vert \vert_{W_d} \\
    &\geq \vert \vert y_{\texttt{rb}}(\mu) \vert \vert_{W_d} - \RBestFineABS(\mu) \\&= \left(1 - \frac{1}{2}\RBestFineREL(\mu)\right) \vert \vert y_{\texttt{rb}}(\mu) \vert \vert_{W_d} \\
    &\geq \frac{1}{2} \vert \vert y_{\texttt{rb}}(\mu) \vert \vert_{W_d} .
\end{align}
With this we finally receive
\begin{align}
    \RBestFineREL(\mu) &= \frac{2\vert \vert \Tilde{r}_d(\mu) \vert \vert _{W_d}}{\alpha(\mu)\vert \vert y_{\texttt{rb}}(\mu) \vert \vert _{W_d}} = 2 \frac{\vert \vert y_d(\mu) \vert \vert_{W_d}}{\vert \vert y_{\texttt{rb}}(\mu) \vert \vert_{W_d}} \frac{\RBestFineABS(\mu)}{\vert \vert y_d(\mu) \vert \vert_{W_d}}\\
    &\geq \frac{\RBestFineABS(\mu)}{\vert \vert y_d(\mu) \vert \vert_{W_d}} \geq \frac{\vert \vert y_{d}(\mu) - y_{\texttt{rb}}(\mu) \vert \vert _{W_d}}{\vert \vert y_{d}(\mu) \vert \vert _{W_d}} .
\end{align}
\end{proof}

\paragraph*{Fully practical estimation of the residual norm.}
The estimators in Theorem \ref{thm:RBerror_ex} have exactly the same formulation as the ones for the elliptic case in \cites[Prop.~4.4]{Hesthaven2016}[Prop.~2.24, 2.27]{haasdonk2017reduced}. However, in comparison the norm of the Riesz representer of the residual $\vert \vert \Tilde{r}_d(\mu) \vert \vert _{W_d}$ is not offline-online decomposable and requires further systems solves with $A_x(\mu)$. In the following we present a variant of offline-online decomposable error estimators, which involve an estimation of the residual norm. \\

To obtain these estimators, we compute the norm of the Riesz representer of the residual and multiply the matrix formulation of the inner product on both sides with $(I_M \otimes A_x(\mu)M_x^{-1})(I_M \otimes A_x(\mu)M_x^{-1})^{-1}$, where $I_M \in \RR^{M\times M}$ is the identity matrix. This leads to a matrix $\mathfrak{S}(\mu) \in \RR^{NM \times NM}$ and vector $\mathfrak{s}(\mu) \in \RR^{NM}$ that allow a rewriting of $\vert \vert \Tilde{r}_d(\mu) \vert \vert _{W_d}$ in Theorem \ref{thm:RB_error_est_cc_sq}. We introduce those quantities in the following definition.
\begin{definition}
Let $I_M \in \RR^{M\times M}$ and $I_N \in \RR^{N\times N}$ denote identity matrices of respective sizes. For fixed $\mu \in \mathcal{P}$ we define
\begin{align}
    \mathfrak{S}(\mu) &:= Z_t^T (M_t^\psi)^{-1} Z_t \otimes M_x + M_t\otimes A_x(\mu) M_x^{-1} A_x(\mu) + T_t\otimes A_x(\mu), \\
    \mathfrak{s}(\mu) &:= R_0^t\otimes A_x(\mu)M_x^{-1} R_0^x(\mu) + \left(I_M \otimes A_x(\mu)M_x^{-1}\right) F_1(\mu) + \left(Z_t^T (M_t^\psi)^{-1} \otimes I_N\right) F_2(\mu).
\end{align}
To obtain an offline-online decomposition of the form
\begin{align}
    \mathfrak{S}(\mu) = \sum_{q=1}^{Q_{\mathfrak{S}}} \theta_{\mathfrak{S}}^{q}(\mu) \mathfrak{S}_{q} \quad \text{and} \quad \mathfrak{s}(\mu) = \sum_{q=1}^{Q_{\mathfrak{s}}} \theta_{\mathfrak{s}}^{q}(\mu) \mathfrak{s}_{q}
\end{align}
we set
\begin{alignat}{3}\label{eq:offlineoperatorstart}
\theta_{\mathfrak{s}}^{q} &:= \theta_{A}^{i} \theta_{y_0}^j, &&\mathfrak{s}_{q} := R_{0}^{t} \otimes A_{x}^{i}M_{x}^{-1}R_{0}^{x,j}, &&q = (j-1) Q_A + i , \\
& && && i = 1, ..., Q_{A}, \, j=1,...,Q_y, \nonumber \\
\theta_{\mathfrak{s}}^{q} &:= \theta_{A}^{i} \theta_{f}^{j}, \quad &&\mathfrak{s}_{q} := \left(I_M \otimes A_{x}^{i} M_{x}^{-1}\right) F_{1}^{j}, \quad &&q = (j-1)Q_{A} + i + Q_{A}Q_y, \\
& && && i = 1, ..., Q_{A}, \, j = 1, ..., Q_{f}, \nonumber \\
\theta_{\mathfrak{s}}^{q} &:= \theta_{f}^{i}, &&\mathfrak{s}_{q} := \left(Z_t^T (M_t^\psi)^{-1} \otimes I_N\right) F_{2}^{i}, \quad &&q = i + Q_{A}Q_{f} + Q_{A}Q_y, \\
& && && i = 1, ..., Q_{f}, \nonumber \\
\theta_{\mathfrak{S}}^{1} &:= 1, &&\mathfrak{S}_{1} := Z_t^T (M_t^\psi)^{-1} Z_t \otimes M_{x}, && \\
\theta_{\mathfrak{S}}^{q} &:= \theta_{A}^{i}\theta_{A}^{j}, \quad && \mathfrak{S}_{q} := M_{t} \otimes A_{x}^{i} M_{x}^{-1} A_{x}^{j}, \quad &&q = (j-1)Q_{A} + i + 1,\\
& && &&i,j = 1, ..., Q_{A}, \nonumber \\
\theta_{\mathfrak{S}}^{q} &:= \theta_{A}^{i}, && \mathfrak{S}_{q} := T_{t} \otimes A_{x}^{i}, &&q = i + Q_{A}^{2} + 1,\label{eq:offlineoperatorend}\\
& && &&i = 1, ..., Q_{A}, \nonumber
\end{alignat}
as well as $Q_{\mathfrak{s}} := Q_{A}Q_y + Q_{A}Q_{f} + Q_{f}$ and $Q_{\mathfrak{S}} := 1 + Q_{A}^{2} + Q_{A}$.
\end{definition}

\begin{theorem}[Estimation of the residual norm]\label{thm:RB_error_est_cc_sq} 
Let $\mu \in \mathcal{P}$ be fixed and $\mathfrak{r}(\mu) \in \RR^{Q_{\mathfrak{s}}+Q_{\mathfrak{S}}\texttt{L}}$, $\mathfrak{R} \in \RR^{(Q_{\mathfrak{s}}+Q_{\mathfrak{S}}\texttt{L}) \times NM}$, $\mathfrak{G} \in \RR^{NM\times NM}$ with
\begin{align}\label{eq:residualRBparameterpart}
    \mathfrak{r}(\mu) &:= \left(\theta_{\mathfrak{s}}^{1}(\mu), ..., \theta_{\mathfrak{s}}^{Q_{\mathfrak{s}}}(\mu), -\Vec{u}_{y}^T\theta_{\mathfrak{S}}^{1}(\mu), ..., -\Vec{u}_{y}^T\theta_{\mathfrak{S}}^{Q_{\mathfrak{S}}}(\mu)\right)^T , \\
    \mathfrak{R} &:= \left(\mathfrak{s}_{1}, ...,  \mathfrak{s}_{Q_{\mathfrak{s}}}, \mathfrak{S}_{1}B_W, ..., \mathfrak{S}_{Q_{\mathfrak{S}}}B_W\right) , \\
    \mathfrak{G} &:= \mathfrak{R}^T\left(Z_t^T (M_t^\psi)^{-1} Z_t \otimes A_{x}(\overline{\mu}) + M_{t} \otimes M_{x}^{-2} A_{x}(\overline{\mu})^3 + T_t \otimes M_x^{-1}A_x(\overline{\mu})^2\right)^{-1}\mathfrak{R} .\label{eq:residualRBotherpart}
\end{align} 
Then it holds that
\begin{alignat}{3}\label{eq:rbabsestnew}
    &\epsilon^\text{abs}(\mu) &&\leq \frac{\sqrt{\mathfrak{r}(\mu)^T \mathfrak{G} \mathfrak{r}(\mu)}}{\mathfrak{c}_c(\mu) \alpha(\mu)} &&=: \RBestCcABS(\mu) , \\
    &\epsilon^\text{rel}(\mu) &&\leq \frac{2 \sqrt{\mathfrak{r}(\mu)^T \mathfrak{G} \mathfrak{r}(\mu)}}{\mathfrak{c}_c(\mu) \alpha(\mu)\vert \vert y_{\texttt{rb}}(\mu) \vert \vert _{W_d}} &&=: \RBestCcREL(\mu), \quad \text{if} \; \; \RBestCcREL(\mu) \leq 1 .
\end{alignat}
\end{theorem}
\begin{proof}
Let $w_i$ denote the $i$-th basis function of $W_d$ and $I_{M} \in \RR^{M \times M}$ an identity matrix. We set 
\begin{align}
    \Vec{r} &:= \left(l_d(\mu; w_{i}) - b_d(\mu; y_{\texttt{rb}}(\mu), w_{i})\right)_{i = 1}^{NM} \in \RR^{NM},\\
    \hat{\mathfrak{r}}(\mu) &:= \left(I_{M} \otimes A_{x}(\mu)M_{x}^{-1}\right)\Vec{r} = \mathfrak{s}(\mu) - \mathfrak{S}(\mu) B_W \Vec{u}_y.\label{eq:residual_vec_to_assemble}
\end{align}
Then it follows that
\begin{align}
    \vert \vert \Tilde{r}_{d}(\mu) \vert \vert^{2}_{W_{d}} &= b_d(\overline{\mu}; \Tilde{r}_{d}(\mu), \Tilde{r}_{d}(\mu))\\
    &= \Vec{r}^T (Z_t^T (M_t^\psi)^{-1} Z_t \otimes M_{x}A_{x}^{-1}(\overline{\mu})M_{x} + M_{t} \otimes A_{x}(\overline{\mu}) + T_t\otimes M_x)^{-1}\Vec{r}\\
\begin{split}
    &= \hat{\mathfrak{r}}(\mu)^T \left(I_{M} \otimes A_{x}(\mu)M_{x}^{-1}\right)^{- T} (Z_t^T (M_t^\psi)^{-1} Z_t \otimes M_{x}A_{x}^{-1}(\overline{\mu})M_{x} \\&\qquad\qquad\qquad\qquad+ M_{t} \otimes A_{x}(\overline{\mu}) + T_t\otimes M_x)^{-1} \left(I_{M} \otimes A_{x}(\mu)M_{x}^{-1}\right)^{-1} \hat{\mathfrak{r}}(\mu) 
\end{split}
    \\
\begin{split}
    &= \hat{\mathfrak{r}}(\mu)^T (Z_t^T (M_t^\psi)^{-1} Z_t \otimes A_{x}(\mu) A_{x}^{-1}(\overline{\mu}) A_{x}(\mu) + M_{t} \otimes A_{x}(\mu) M_{x}^{-1} A_{x}(\overline{\mu}) M_{x}^{-1} A_{x}(\mu) \\&\qquad\qquad\qquad\qquad+ T_t \otimes A_x(\mu)M_x^{-1}A_x(\mu))^{-1}\hat{\mathfrak{r}}(\mu)
\end{split} 
    \\
\begin{split}\label{eq:estmofrieszresid}
    &\leq \hat{\mathfrak{r}}(\mu)^T (Z_t^T (M_t^\psi)^{-1} Z_t \otimes A_{x}(\overline{\mu}) + M_{t} \otimes A_{x}(\overline{\mu}) M_{x}^{-1} A_{x}(\overline{\mu}) M_{x}^{-1} A_{x}(\overline{\mu}) \\&\qquad\qquad\qquad\qquad+ T_t \otimes A_x(\overline{\mu})M_x^{-1}A_x(\overline{\mu}))^{-1}\hat{\mathfrak{r}}(\mu) (\mathfrak{c}_c(\mu))^{-2}.
\end{split} 
\end{align}
The inequality in \eqref{eq:estmofrieszresid}, where we estimate terms involving $\mu$ with corresponding terms involving the reference parameter $\overline{\mu}$, follows by replacing
\begin{align}
    A_x(\mu) = A_x(\mu)A_x(\overline{\mu})^{-1}A_x(\overline{\mu})
\end{align}
two times and applying the property
\begin{align}
    \Vec{v}^T \left(A_x(\mu)\right)^{-1} \Vec{v} &= \Vec{v}^T \left(A_x(\mu)A_x(\overline{\mu})^{-1}A_x(\overline{\mu})\right)^{-1} \Vec{v} \\&\leq \lambda_{\max}\left(\left(A_x(\mu)A_x(\overline{\mu})^{-1}\right)^{-1}\right) \Vec{v}^T A_x(\overline{\mu})^{-1} \Vec{v} \\
    &= \lambda_{\min}\left(A_x(\overline{\mu})^{-1} A_x(\mu)\right)^{-1} \Vec{v}^T A_x(\overline{\mu})^{-1} \Vec{v} \\
    &= \mathfrak{c}_c(\mu)^{-1} \Vec{v}^T A_x(\overline{\mu})^{-1} \Vec{v}
\end{align}
for any $\Vec{v} \in \RR^N$, where $\lambda_{\max}$ and $\lambda_{\min}$ denote the largest resp.~smallest eigenvalues. This yields the claim for the absolute estimator. 

Assuming $\RBestCcREL(\mu) \leq 1$, the assumptions of Theorem \ref{thm:RBerror_ex} are fulfilled and the claim for the relative estimator follows.
\end{proof}
\begin{remark} The estimation in \eqref{eq:estmofrieszresid} is fully offline-online decomposable analogously to the residual of elliptic PDEs \cite[§~4.2.5]{Hesthaven2016}. Furthermore, in the standard finite element setting of Remark \ref{rem:stiffnessTime} \textit{mass lumping} can be used for approximating $M_x^{-1}$ in practice. 
\end{remark}

\begin{remark}
To avoid the computation of a continuity constant, $\overline{\mu} \in \mathcal{P}$ can be chosen such that $\mathfrak{c}_c(\mu)\mathfrak{c}_s(\mu) \leq 1$ holds for all parameters $\mu$, which guarantees 
\begin{align}
    \alpha(\mu) = \mathfrak{c}_c(\mu) .
\end{align}
Alternatively, \eqref{eq:normsEnergyEqui2} can be used to estimate $\alpha(\mu)$ with $\alpha_\text{LB}$.
\end{remark}

In practice, the coercivity constant can be calculated by a (multi-parameter) min-theta-approach \cites[§4.2.2]{RozzaRB}[§~4.3.2]{Hesthaven2016} or the successive constraint method from \cite{CRMATH_2007__345_8_473_0} effectively.

\begin{remark}
From \eqref{eq:residual_vec_to_assemble} we can motivate that for the coefficient vector $\Vec{y}$ of the high fidelity solution it holds that
\begin{align}
    \mathfrak{S}(\mu)\Vec{y} = \mathfrak{s}(\mu) .
\end{align}
We also notice that $\mathfrak{S}(\mu)$ defines a \textit{discrete operator} on $W_d$ that is uniformly coercive and continuous. Additionally, as $\mathfrak{S}(\mu)$ is parameter-separable, the assumptions of \cite[Thm.~3.1]{ohlberger} are fulfilled and the convergence in the discrete $W_d$ norm follows. In particular we receive by \cite[Thm.~3.1]{ohlberger} that the Kolmogorov $\texttt{L}$-width decays exponentially, if the reduced basis spaces are constructed in a meaningful way (e.g.~POD). Therefore we can expect that the worst best-approximation error exponentially tends to zero for small $\texttt{L}$ already.
\end{remark}

\section{Numerical examples.}\label{subsec:numericsRB}
We show the performance of the reduced basis approach with a POD-greedy algorithm \cites{SIENA2023127}{haasdonkCONVrb} for two example problems. With $I_k$ we denote the time grid on $I = (0,T]$, and $\Omega_h$ denotes a conforming triangulation on the spatial domain $\Omega$. Furthermore, we denote with $d$ the space-time mesh size parameter composed of the time grid size $k$ of $I_k$ and the space grid width $h$ of $\Omega_h$ according to $d^2 = k^2 + h^2$. We choose $V_N$ in \eqref{VN} as the space of piecewise linear and globally continuous finite elements (CG~1) defined over $\Omega_h$, and use mass lumping for the associated mass matrix as described in \cite[§~5.1.2, (1.20)]{GrossmannRoos}. With respect to time we take piecewise linear and globally continuous elements for $K_M$ in \eqref{KM} and piecewise constant and discontinuous elements (DG~0) for $J_P$ in \eqref{JP}. An efficient preconditioner for the solution of the corresponding linear systems is provided in \cite{preprint_SP}.

We run our simulations on a machine with an \texttt{AMD Ryzen Threadripper PRO 5995WX} CPU and 512 GB RAM. We use the \texttt{IPython} interpreter \cite{PER-GRA:2007} with \texttt{dolfinx} \cites{BarattaEtal2023}{BasixJoss}{ScroggsEtal2022}{AlnaesEtal2014} for assembling the finite element matrices. The linear systems and eigenvalue problems are solved using the \texttt{SciPy} library \cite{2020SciPy-NMeth} and \texttt{PETSc} \cites{petsc1}{petsc2}.

We consider two numerical examples below. The first is a standard \textit{thermal block} problem adapted from the literature, and the second is an example with minimal regularity.

\paragraph*{Example 1: Two-dimensional thermal block problem.} Thermal block problems are investigated e.g.~in \cites[§~2.2.1, §~3.5.2]{RozzaRB}[§~2.3.1]{haasdonk2017reduced}{Rozza2008}[§~6.1.4]{Hesthaven2016}. In the following we use the parabolic formulation from \cite{Hesthaven2016}. We aim at comparing the absolute error estimators from Theorems \ref{thm:RBerror_ex} and \ref{thm:RB_error_est_cc_sq}.\\

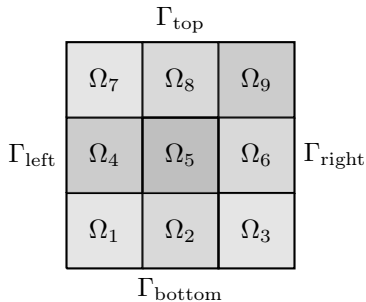
\begin{figure}%
    \centering
    \begin{tikzpicture}[x=0.5cm,y=0.5cm]
\coordinate (A) at (0, 0);
\coordinate (B) at (4, 0);
\coordinate (C) at (4, 4);
\coordinate (D) at (0, 4);

\coordinate (E) at (0, 2);
\coordinate (F) at (2, 0);
\coordinate (G) at (4, 2);
\coordinate (H) at (2, 4);
\coordinate (I) at (2, 2);

\coordinate (J) at (6, 0);
\coordinate (K) at (6, 2);
\coordinate (L) at (6, 4);
\coordinate (M) at (6, 6);
\coordinate (N) at (4, 6);
\coordinate (O) at (2, 6);
\coordinate (P) at (0, 6);

\draw[thick, black, fill=gray, fill opacity=0.2] (A) -- (F) -- (I) -- (E) -- (A);
\draw[thick, black, fill=gray, fill opacity=0.3] (I) -- (G) -- (B) -- (F);
\draw[thick, black, fill=gray, fill opacity=0.4] (I) -- (H) -- (D) -- (E);
\draw[thick, black, fill=gray, fill opacity=0.5] (I) -- (H) -- (C) -- (G);
\draw[thick, black, fill=gray, fill opacity=0.2] (D) -- (P) -- (O) -- (H);
\draw[thick, black, fill=gray, fill opacity=0.3] (O) -- (N) -- (C) -- (H);
\draw[thick, black, fill=gray, fill opacity=0.4] (N) -- (M) -- (L) -- (C);
\draw[thick, black, fill=gray, fill opacity=0.3] (L) -- (K) -- (G) -- (C);
\draw[thick, black, fill=gray, fill opacity=0.2] (K) -- (J) -- (B) -- (G);


\node[] at (1, 1) {$\Omega_1$};
\node[] at (3, 1) {$\Omega_2$};
\node[] at (5, 1) {$\Omega_3$};

\node[] at (1, 3) {$\Omega_4$};
\node[] at (3, 3) {$\Omega_5$};
\node[] at (5, 3) {$\Omega_6$};

\node[] at (1, 5) {$\Omega_7$};
\node[] at (3, 5) {$\Omega_8$};
\node[] at (5, 5) {$\Omega_9$};

\node[anchor=south] at (3,6) {$\Gamma_\text{top}$};
\node[anchor=north] at (3,0) {$\Gamma_\text{bottom}$};
\node[anchor=east] at (0,3) {$\Gamma_\text{left}$};
\node[anchor=west] at (6,3) {$\Gamma_\text{right}$};

\end{tikzpicture}
    \caption{Sketch of $\Omega$ and its subdomains and boundaries in the thermal block example.}
    \label{fig:sketchEXNM}
\end{figure}

The domain $\Omega:= (0,1)^2$ is subdivided into nine equally sized subdomains $\Omega_1, ..., \Omega_9$,  as illustrated in Figure~\ref{fig:sketchEXNM}.
To the subdomains $\Omega_p$, $p=1,...,8$, we assign the diffusivity constant $\mu_p$ $(p=1,...,8)$, resulting in the parameter vector $(\mu_p)_{p=1}^8$. We consider \eqref{eq:parab_mu} with $y_0=0$, where the bilinear form in \eqref{eq:RBbilinA} is given by
\begin{equation}
    a(\mu;u,v) = \sum_{p=1}^8 \mu_p \int_{\Omega_p} \nabla u \cdot \nabla v + \int_{\Omega_9} \nabla u \cdot \nabla v .
\end{equation}
On $\Gamma_\text{top}$ homogeneous Dirichlet and on $\Gamma_\text{left}$ and $\Gamma_\text{right}$ homogeneous Neumann boundary conditions are prescribed. On $\Gamma_\text{bottom}$ we consider parameterized Neumann data, yielding the right hand side
\begin{equation}
    f(\mu; w) = \mu_9 \int_0^T \int_{\Gamma_\text{bottom}} w
\end{equation}
in \eqref{eq:parab_mu}.
As in \cite{Hesthaven2016} we set $\mathcal{P} = [0.1,10]^8 \times [-1,1]$ and consider $\mu := ((\mu_p)_{p=1}^8,\mu_9)\in \mathcal P$. The training set $\mathcal{S}_{\text{train}}$ consists of $5000$ randomly chosen parameters from $\mathcal P$, where $(\mu_p)_{p=1}^8$ are taken from a uniform distribution on a log scale, and where $\mu_9$ is taken from a uniform distribution on a linear scale. This is motivated in \cites[§3.5.2]{RozzaRB}. As reference parameter we take $\overline{\mu} = (1,...,1)^T$. We set $T=3$ and use a time grid with stepsize  $k=0.05$, yielding $M = 60$ and $P=59$. In space we a choose an equidistant, symmetric and regular triangulation $\Omega_h$ of $\Omega$ with 22 vertices in each spatial direction, yielding $N = 22^2 = 484$.\\

We use the POD-greedy approach from Algorithm \ref{greedy} with $\texttt{L}_1 = 1$, $\texttt{L}_2 = 2$ together with the estimator $\RBestCcABS$ from \eqref{eq:rbabsestnew}. To generate a sequence of reduced basis spaces $W_\texttt{L}$, $Q_\texttt{L}$,
we solve the reduced system \eqref{eq:reduziertesproblem} for all unused parameters in the training set $\mathcal{S}_{\text{train}}$. We then evaluate the error estimator $\RBestCcABS$ for all reduced solutions
and choose iteratively $L_2$ parameters that maximize the estimator. We solve the high fidelity problem \eqref{eq:hfproblemSP} for those parameters. The reduced basis space then is enriched by these high fidelity solutions and applying POD, see Algorithm \ref{greedy}.

For comparison, we generate a validation set of 20 randomly chosen parameters, not included in $\mathcal{S}_{\text{train}}$, and evaluate the error estimators $\RBestCcABS$ and $\RBestFineABS$ as well as the true error in each iteration. The computational costs of the true error evaluation limits the size of the validation set. To allow a better assessment of the results, we also compute the effectivities of the estimators. For a given estimator $\eta$ we denote by
\begin{equation}
    \eff (\eta):= \frac{\eta}{\epsilon^\text{abs}}
\end{equation}
its effectivity, which we calculate for all parameters in the validation set. \\

In Figure~\ref{fig:rberrorNM} (left) we show the average error on the validation set in dependence of the reduced basis space dimension $\texttt{L}$ in a semi-logarithmic plot. In Figure~\ref{fig:rberrorNM} (right) we show the average effectivities for the two estimators $\RBestFineABS$ and $\RBestCcABS$ against $\texttt{L}$.

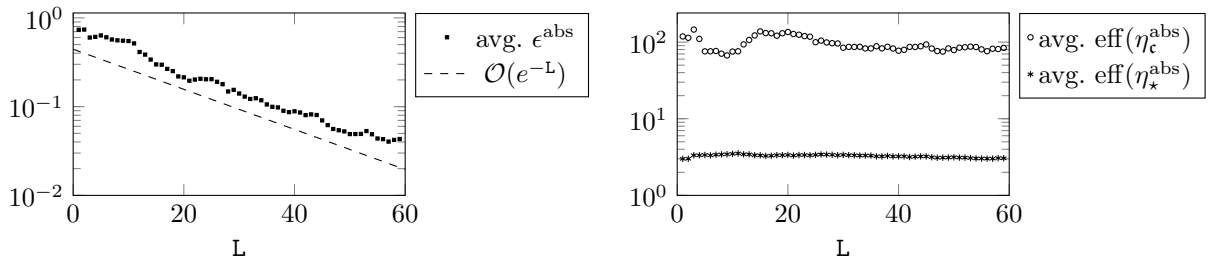
\begin{figure}
    \centering
    \begin{tikzpicture} 
\pgfplotstableread[col sep = comma]{RB/resultsNM.csv}\tabNMa;

\begin{axis}
[
width=0.35\textwidth,
height=4cm,
xlabel={$\texttt{L}$},
ylabel={ },
xmin = 0,
xmax = 60,
ymin = 0.01,
ymode =log,
xmode=normal,
legend pos = outer north east,
]




\addplot[only marks, mark=square*, mark size=0.6pt] table [x index = 0, y index = 1] from \tabNMa;
\addlegendentry{avg.~$\epsilon^\text{abs}$}

\addplot[dashed,samples =  3, domain=0:60]{0.7308*e^(-0.052*x - 0.5)};
\addlegendentry{$\mathcal{O}(e^{-\texttt{L}})$}

\end{axis}

\end{tikzpicture} \, \begin{tikzpicture} 
\pgfplotstableread[col sep = comma]{RB/resultsNM.csv}\tabNMa;

\begin{axis}
[
width=0.35\textwidth,
height=4cm,
xlabel={$\texttt{L}$},
ylabel={ },
xmin = 0,
xmax = 60,
ymin = 1,
ymode =log,
xmode=normal,
legend pos = outer north east,
]

\addplot[only marks, mark=o, mark size=1pt] table [x index = 0, y index = 2] from \tabNMa;
\addlegendentry{avg.~$\eff(\RBestCcABS)$}

\addplot[only marks, mark=asterisk, mark size=1.2pt] table [x index = 0, y index = 3] from \tabNMa;
\addlegendentry{avg.~$\eff(\RBestFineABS)$}

\end{axis}

\end{tikzpicture}
    \caption{\textit{Left:} Absolute error of the reduced basis space $W_\texttt{L}$, as evaluated on a validation set, plotted against the dimension $\texttt{L}$ of the reduced basis space. As in a classical reduced basis approach, we observe exponential decay. \textit{Right:} Effectivities of the error estimators $\RBestFineABS$ (from Theorem \ref{thm:RBerror_ex}) and $\RBestCcABS$ (from Theorem \ref{thm:RB_error_est_cc_sq}), computed on the validation set, too. The estimator $\RBestFineABS$, which evaluates the true residual, produces the most precise results due to effectivities close to 1, whereas the estimator $\RBestCcABS$ uses an estimation of the residual.}
    \label{fig:rberrorNM}
\end{figure}

As in a classical reduced basis approach for elliptic PDEs, we observe exponential decay in the errors. We stress that the error is measured with respect to the full validation set that represents the whole parameter space $\mathcal{P}$. For an investigation of the maximal predicted error for single parameters $\mu$ we refer to the next example. In Figure~\ref{fig:rberrorNM} we also see that the estimator $\RBestFineABS$, involving the exact residual, shows effectivities close to 1, which is the optimal value. The estimator $\RBestCcABS$ performs slightly worse. However, results in this order are very comparable to those reported in \cite{Hesthaven2016}. We mention, that we certify in the stronger $W_d$ norm, which also considers the time derivative in $L^2(0,T;V^\star)$. The latter work provides a certification in $L^2(0,T;V)$. The effectivities and required reduced basis dimensions for a desired error reduction are still in a similar range, although the stronger norm is used in the present work. \\

In conclusion, the estimator $\RBestFineABS$ gives more accurate estimations than $\RBestCcABS$. However, $\RBestFineABS$ comes at higher computational costs than the other estimator. In practical applications a compromise might be to use different estimators for parameter selection and a posteriori certification, which we show in the next example.

\paragraph*{Example 2: A three-dimensional problem with minimal regularity.} We now apply the reduced basis approach to an example in three spatial dimensions with minimal regularity. This enables us to study the practicality of error estimators and computational time for a more complex problem. 
We use $I=(0,1]$ and set
\begin{align}
    \Omega &:= \left( (0,1)^2 \setminus ([0.5, 1] \times [0, 0.5])\right)\times (0, 0.5), \\
    \Omega_1 &:= \lbrace x \in \Omega \; \vert \; (x_1-0.25)^2 + (x_2-0.25)^2 < 0.2^2, \; x_3 < 0.2 \rbrace , \\
    \Omega_2 &:= \lbrace x \in \Omega \; \vert \; (x_1-0.25)^2 + (x_2-0.75)^2 < 0.2^2, \; x_3 < 0.2 \rbrace , \\
    \Omega_3 &:= \lbrace x \in \Omega \; \vert \; (x_1-0.75)^2 + (x_2-0.75)^2 < 0.2^2, \; x_3 < 0.2 \rbrace , \\
    \Omega_0 &:= \Omega \setminus (\Omega_1 \cup \Omega_2 \cup \Omega_3) .
\end{align}
The situation is sketched in Figure \ref{fig:sketchEX2}. 
\begin{figure}
    \centering
    \def\rotx{40}
\def\rotz{-15}
\tdplotsetmaincoords{\rotx}{\rotz}
\begin{tikzpicture}[scale=3,tdplot_main_coords]

  \coordinate (A) at (0,0,0);
  \coordinate (B) at (0.5,0,0);
  \coordinate (C) at (0.5,0.5,0);
  \coordinate (D) at (1,0.5,0);
  \coordinate (E) at (1,1,0);
  \coordinate (F) at (0,1,0);

  \coordinate (G) at (0,0,0.5);
  \coordinate (H) at (0.5,0,0.5);
  \coordinate (I) at (0.5,0.5,0.5);
  \coordinate (J) at (1,0.5,0.5);
  \coordinate (K) at (1,1,0.5);
  \coordinate (L) at (0,1,0.5);
  
  \coordinate (M) at (0.25,0.25,0);
  \coordinate (N) at (0.25,0.75,0);
  \coordinate (O) at (0.75,0.75,0);

  \coordinate (P) at (0.25,0.25,0.2);
  \coordinate (Q) at (0.25,0.75,0.2);
  \coordinate (R) at (0.75,0.75,0.2);

  \draw[black, fill=black, fill opacity=0.1] (A) -- (B) -- (C) -- (D) -- (E) -- (F) -- cycle;

  \def\radcirc{0.2}
  \def\cylheight{0.2}
  
  \draw [thick, clUKfb3, fill=white, fill opacity=0.7](M) circle (\radcirc);
  \draw [thick, clUKfb3, fill=white, fill opacity=0.7](N) circle (\radcirc);
  \draw [thick, clUKfb3, fill=white, fill opacity=0.7](O) circle (\radcirc);
  \draw [thick, clUKfb3, fill=white, fill opacity=0.5](P) circle (\radcirc);
  \draw [thick, clUKfb3, fill=white, fill opacity=0.5](Q) circle (\radcirc);
  \draw [thick, clUKfb3, fill=white, fill opacity=0.5](R) circle (\radcirc);
  
  \pgfcoordinate{edge1_top}{ \pgfpointcylindrical{\rotz}{\radcirc}{\cylheight}};
  \pgfcoordinate{edge1_bottom}{ \pgfpointcylindrical{\rotz}{\radcirc}{0}};
  \draw[thick, clUKfb3] ($(edge1_top) + (M)$) -- ($(edge1_bottom) + (M)$);
  \draw[thick, clUKfb3] ($(edge1_top) + (N)$) -- ($(edge1_bottom) + (N)$);
  \draw[thick, clUKfb3] ($(edge1_top) + (O)$) -- ($(edge1_bottom) + (O)$);
  \pgfcoordinate{edge2_top}{ \pgfpointcylindrical{\rotz+180}{\radcirc}{\cylheight} };
  \pgfcoordinate{edge2_bottom}{ \pgfpointcylindrical{\rotz+180}{\radcirc}{0} };
  \draw[thick, clUKfb3] ($(edge2_top) + (M)$) -- ($(edge2_bottom) + (M)$);
  \draw[thick, clUKfb3] ($(edge2_top) + (N)$) -- ($(edge2_bottom) + (N)$);
  \draw[thick, clUKfb3] ($(edge2_top) + (O)$) -- ($(edge2_bottom) + (O)$);

  \draw[dashed] (G) -- (H) -- (I) -- (J) -- (K) -- (L) -- cycle;
  \draw[dashed] (A) -- (G);
  \draw[dashed] (B) -- (H);
  \draw[dashed] (C) -- (I);
  \draw[dashed] (D) -- (J);
  \draw[dashed] (E) -- (K);
  \draw[dashed] (F) -- (L);

  \draw[] (-0.25,0.25) node[text=clUKfb3] {$\Omega_1$};
  \draw[] (-0.25,0.75) node[text=clUKfb3] {$\Omega_2$};
  \draw[] (1.25,0.75) node[text=clUKfb3] {$\Omega_3$};
  \draw[] (0.75,0.25) node[text=black, opacity=0.4] {$\Omega_0$};
  \draw[] (0.5,1.25) node[text=black] {$\Omega$};
\end{tikzpicture}
    \caption{Sketch of the domain $\Omega$ and their subdomains $\Omega_0$, $\Omega_1$, $\Omega_2$, $\Omega_3$ in the second example.}
    \label{fig:sketchEX2}
\end{figure}
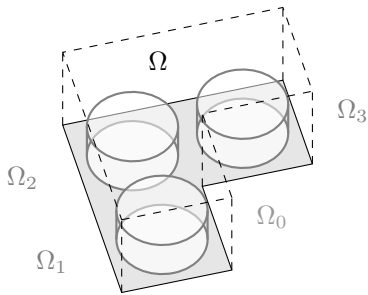
We consider \eqref{eq:parab_mu} with $y_0=0$ and the bilinear form in \eqref{eq:RBbilinA} given by
\[a(\mu;u,v):= \int_\Omega \kappa(\mu)\nabla u \nabla v,
\]
where the diffusion coefficient is defined as
\begin{align}
    \kappa(\mu;\cdot) = \mathbf{1}_{\Omega_0} + \mu_1\mathbf{1}_{\Omega_1} + \mu_2\mathbf{1}_{\Omega_2} + \mu_3\mathbf{1}_{\Omega_3}.
\end{align}
Here, $\mathbf{1}_{\Omega_i}$ denotes the indicator function of the set $\Omega_i$ $(i=1,2,3)$.
As source term we take the space-time separable function $f(\mu) = f^tf^x(\mu) \in L^2(0,T;V^\star)$ with
\begin{align*}
\label{eq:bsprechteseite0}
    f^t(t) := \begin{cases}1, & \text{if } t \leq 0.5, \\ 0, & \text{else},
    \end{cases}
\end{align*}
and
\begin{align*}
    \langle f^x(\mu), v \rangle_{V^\star, V} := \int_\Omega (\mu_4 \mathbf{1}_{\Omega_1} + \mu_5 \mathbf{1}_{\Omega_2} + \mu_6 \mathbf{1}_{\Omega_3})v_{x_{1}} \qquad \forall v \in V.
\end{align*}

For the parameters we take $((\mu_1,\mu_2,\mu_3),(\mu_4,\mu_5,\mu_6)) \in \mathcal{P}:= [0.25, 4]^3 \times [1,3]^3$ and set $\mu:= (\mu_1,\ldots,\mu_6)^T$.
We choose $S_{\text{train}} \subset \mathcal{P}$ as the Cartesian product of three logarithmically-spaced grids with ten grid points in each set, as suggested in \cites[§3.5.2]{RozzaRB}, and three linearly-spaced grids of $[1,3]$ with three points in each direction, hence $\vert S_{\text{train}} \vert = 10^3\cdot 3^3 = 27000$. Again we take $\overline{\mu} = (1,...,1)^T$. We furthermore choose $M = 16$ and a spatial mesh with $N = 1761$ vertices, where the mesh is generated using \texttt{Gmsh} \cite{Geuzaine2009-tj}. For resolving the discontinuities in $f$ and $\kappa$, we work with a triangulation that considers the subdomains. As in Example 1 we take $\texttt{L}_1 = 1$ and $\texttt{L}_2 = 2$.\\

We compare the absolute and relative error estimators $\RBestFineABS$ resp.~$\RBestFineREL$ with the true absolute and relative errors. For that we use the estimators $\RBestCcABS$ resp.~$\RBestCcREL$ to select parameters and $\RBestFineABS$ resp.~$\RBestFineREL$ for the certification. In each iteration we compute the true absolute and relative error for the selected parameter by evaluating the full $W_d$ norm. The results are shown in Figure \ref{fig:rberror2}.

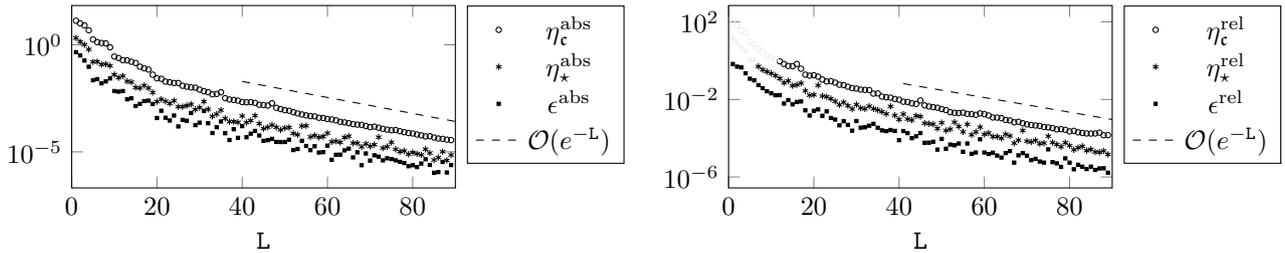
\begin{figure}
    \centering
    \begin{tikzpicture} 
\pgfplotstableread[col sep = comma]{RB/results3D.csv}\tabRBb;

\begin{axis}
[
width=0.39\textwidth,
height=4cm,
xlabel={$\texttt{L}$},
ylabel={ },
xmin = 0,
xmax = 90,
ymode =log,
xmode=normal,
legend pos = outer north east,
]

\addplot[only marks, mark=o, mark size=1pt] table [x index = 0, y index = 2] from \tabRBb;
\addlegendentry{$\RBestCcABS$}

\addplot[only marks, mark=asterisk, mark size=1.2pt] table [x index = 0, y index = 3] from \tabRBb;
\addlegendentry{$\RBestFineABS$}

\addplot[only marks, mark=square*, mark size=0.6pt] table [x index = 0, y index = 1] from \tabRBb;
\addlegendentry{$\epsilon^\text{abs}$}

\addplot[dashed,samples =  3, domain=40:100]{0.067*e^(-0.086*x + 2.2)};
\addlegendentry{$\mathcal{O}(e^{-\texttt{L}})$}

\end{axis}

\end{tikzpicture} \, \begin{tikzpicture} 
\pgfplotstableread[col sep = comma]{RB/results3D.csv}\tabRBb;

\begin{axis}
[
width=0.39\textwidth,
height=4cm,
xlabel={$\texttt{L}$},
ylabel={ },
xmin = 0,
xmax = 90,
ymode =log,
xmode=normal,
legend pos = outer north east,
]

\addplot[only marks, mark=o, mark size=1pt] table [x index = 0, y index = 6] from \tabRBb;
\addlegendentry{$\RBestCcREL$}

\addplot[only marks, mark=asterisk, mark size=1.2pt] table [x index = 0, y index = 8] from \tabRBb;
\addlegendentry{$\RBestFineREL$}

\addplot[only marks, mark=square*, mark size=0.6pt] table [x index = 0, y index = 4] from \tabRBb;
\addlegendentry{$\epsilon^\text{rel}$}

\addplot[dashed,samples=3, domain=41:100]{0.2563*e^(-0.087*x + 2.2)};
\addlegendentry{$\mathcal{O}(e^{-\texttt{L}})$}

\addplot[only marks, mark=asterisk, mark size=1.2pt, color=.!25!white] table [x index = 0, y index = 7] from \tabRBb;

\addplot[only marks, mark=o, mark size=1pt, color=.!25!white] table [x index = 0, y index = 5] from \tabRBb;

\end{axis}

\end{tikzpicture}
    \caption{Comparison of the error estimators $\RBestCcABS$ and $\RBestFineABS$ (left) resp.~$\RBestCcREL$ and $\RBestFineREL$ (right) with the real errors $\epsilon^\text{abs}$ resp.~$\epsilon^\text{rel}$ in dependence of the number $\texttt{L}$ of used basis functions for Example 2. The parameter is freshly selected by $\RBestCcABS$ resp.~$\RBestCcREL$, and is different in every step. We plot the error for the parameter that maximizes the latter error estimators. The region where the assumption $\RBestCcABS \leq 1$ resp.~$\RBestFineREL \leq 1$ from \eqref{eq:RB_first_rel_est} is not satisfied is plotted in gray.}\label{fig:rberror2}
\end{figure}

Similar to a classical reduced basis approach, we observe an exponential decay of the absolute and relative errors and estimators. As expected, there are no compromises in the certification quality of the error estimators, even though this example is posed in three spatial dimensions and only minimal regularity assumptions are fulfilled.\\

To get an impression of the most relevant information that POD provides, we examine the first three basis functions it produces, cutting the three-dimensional spatial domain at $z = 0.1$ and focusing on the lower part only. Figure \ref{fig:basisEX2} shows the basis functions at time $t = 0.4$.

\begin{figure}
    \centering
    \includegraphics[width=0.25\textwidth]{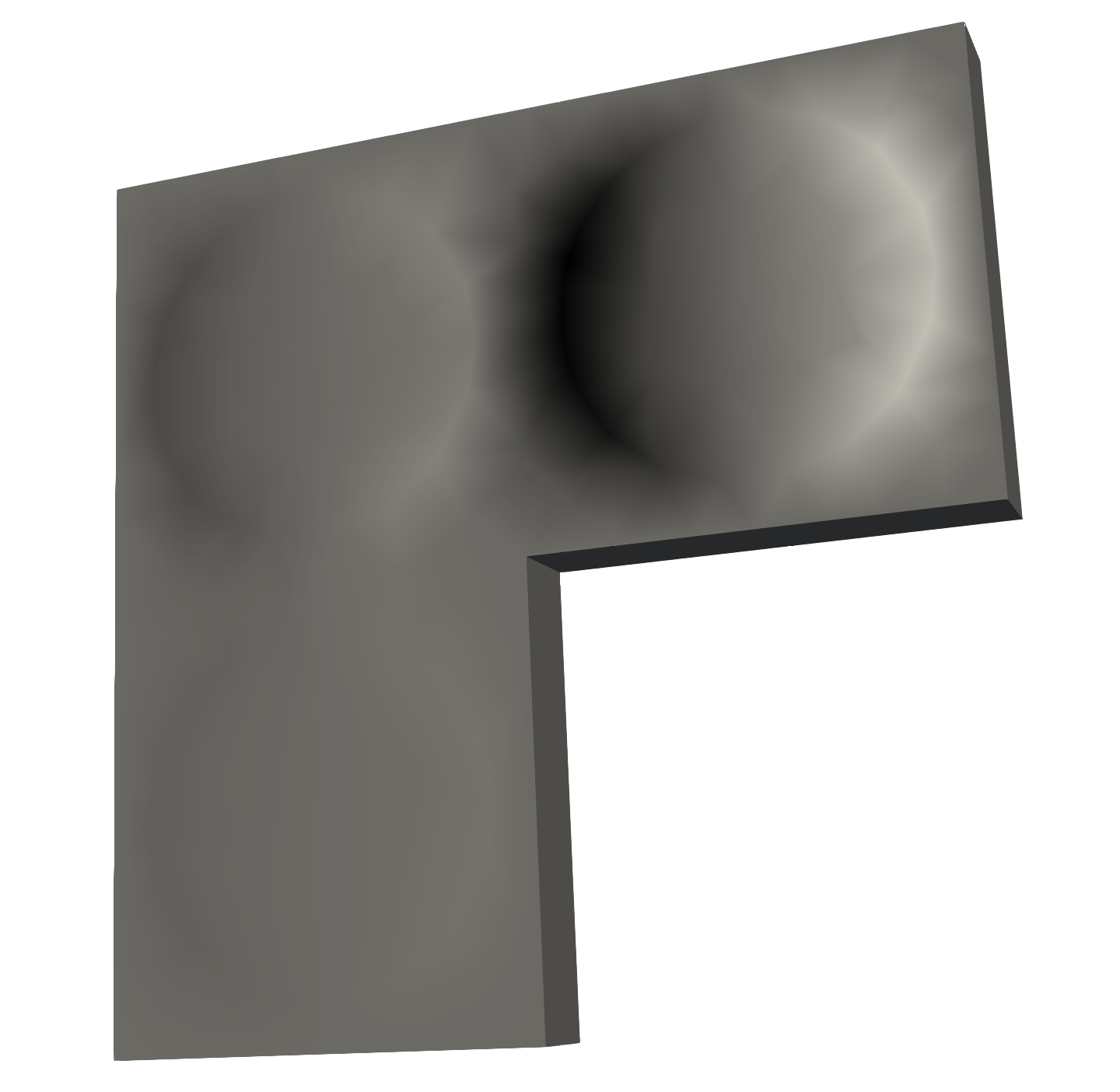}
    \includegraphics[width=0.25\textwidth]{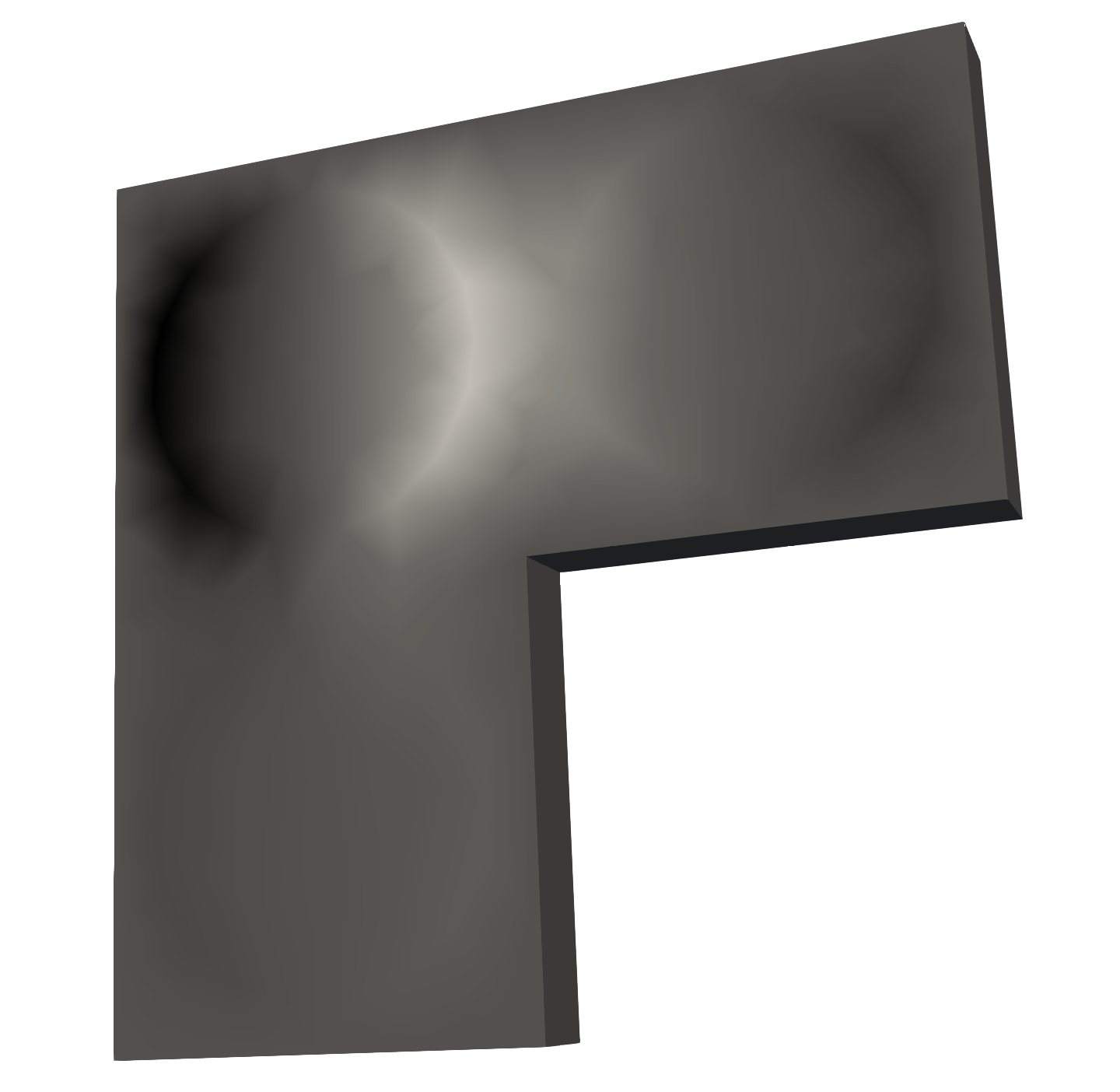}
    \includegraphics[width=0.25\textwidth]{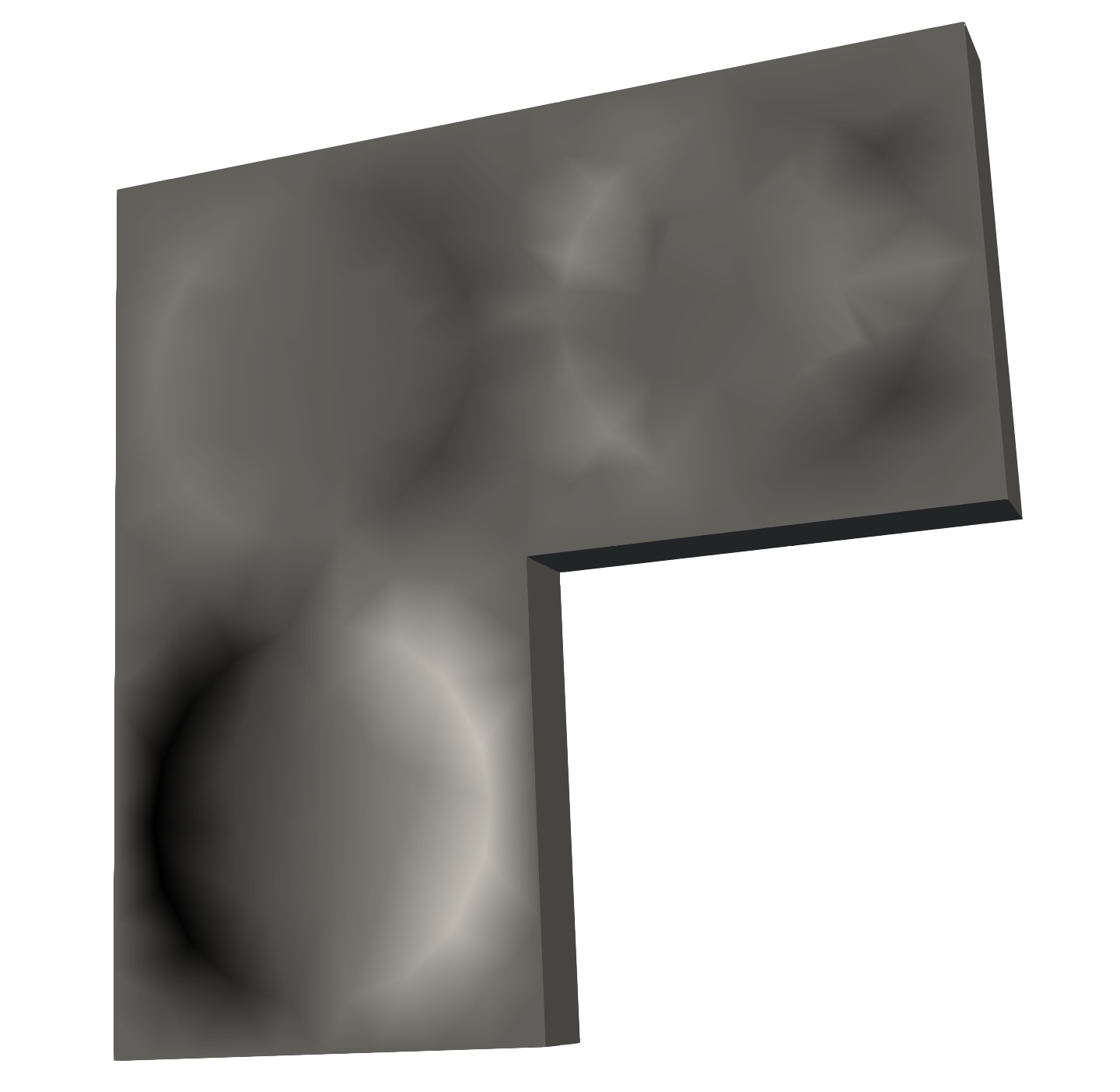}
    \caption{The first three basis functions $\xi_1$, $\xi_2$, $\xi_3$ generated by Algorithm \ref{greedy} for the second example. The domain is cut at $z = 0.1$ and the functions are plotted at $t = 0.4$. We choose a different color scaling for each function in order to highlight the differences and illustrate the most important information. White means small values, black means large values.}\label{fig:basisEX2}
\end{figure}

We see that the most pronounced structure is located at the boundaries of the subdomains $\Omega_1$, $\Omega_2$ and $\Omega_3$, where the jump discontinuities of diffusivity and heat source are located. Also, the direction of the POD-oscillation is in $x_1$-direction which can be expected by the choice of $f^x$. \\

Additionally, we measure and compare the computational times of the offline and online phases to a full system solve without the reduced basis approach, in order to validate that the reduced basis approach saves time. The measured computation times are shown in Figure \ref{fig:calculation_time}.

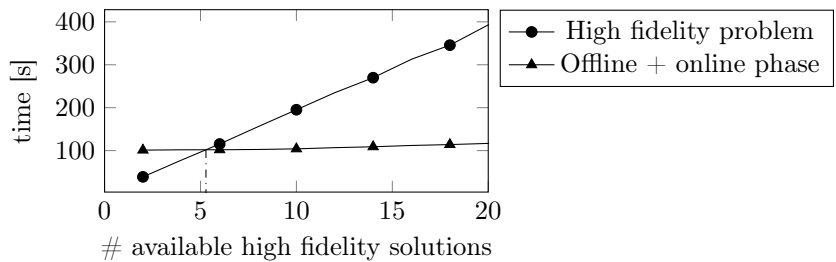
\begin{figure}
    \centering
    \begin{tikzpicture} 
\pgfplotstableread[col sep = comma]{RB/calc_times.csv}\tabCalcTimes;

\begin{axis}
[
width=0.39\textwidth,
height=4cm,
xlabel={\# available high fidelity solutions},
ylabel={time [s]},
xmin = 0,
xmax = 20, 
ymode =normal,
xmode=normal,
legend pos = outer north east,
]

\addplot[thin, mark=*, mark repeat=2, mark phase=0, mark size=2pt] table [x index = 0, y index = 1] from \tabCalcTimes;
\addlegendentry{High fidelity problem}

\addplot[thin, mark=triangle*, mark repeat=2, mark phase=0, mark size=2.2pt] table [x index = 0, y index = 2] from \tabCalcTimes;
\addlegendentry{Offline + online phase}


\draw[dashdotted] (axis cs:5.285530332,-100) -- (axis cs:5.285530332,101.9163919);

\end{axis}

\end{tikzpicture}
    \caption{We compare the computation times for high-fidelity system solutions that are computed and used for the reduced basis approach. For the latter time measurement, we consider the computation times of the offline and online phases, in which the previous solutions are stored in $\mathcal{Z}$. We observe that the additional cost in the offline and online phase pays off after six system solves already.}\label{fig:calculation_time}
\end{figure}

We highlight that the time-measurements strongly depend on the implementation. Here, we are interested in a proof of concept and therefore use the standard routines without exploiting parallel computation in all programming steps. Figure \ref{fig:calculation_time} should therefore be seen in a qualitative fashion. We observe that the runtime of the full problem is more expensive than the reduced basis approach from six high fidelity solves onward. To increase the accuracy of the reduced basis approach, we store the previously computed high fidelity solutions in the set $\mathcal{Z}$, so that they are available for Algorithm \ref{greedy}. As POD is applied on the full set $\mathcal{Z}$, the computation time of the offline phase increases slightly with more solutions available. Furthermore, we increase the reduced basis space dimension $\texttt{L}$ in each step to obtain a higher accuracy. Therefore, building the matrix $\mathfrak{G}$ in \eqref{eq:residualRBotherpart} as well as $B_Q$ in the offline phase also becomes more expensive with larger $\texttt{L}$. However, we see that after six high fidelity system solves, any additional system solve is more costly than running the offline phase and full online phase.

\section{Conclusion.}
In this work we propose a POD-greedy reduced basis method for parabolic equations based on the least squares space-time approach from \cite{preprint_SP}. We adapt this formulation to a parameter-dependent setting. The resulting variational formulation gives rise to a uniformly coercive and continuous bilinear form, for which well-known reduced basis techniques for parametrized elliptic equations can be applied. For certification we propose absolute and relative error estimators in a discrete $W(0,T)$ norm, which can be efficiently evaluated using standard offline–online decomposition techniques. We illustrate the performance of our approach using two numerical examples.

\section*{Statements and Declarations.}

\subsubsection*{Conflict of Interest.}
The authors declare that they have no conflict of interest.

\subsubsection*{Funding.} 
The first author acknowledges funding of the project \textit{Ein nichtglatter Phasenfeld Zugang für Formoptimierung mit instationären Fluiden} by the German Research foundation within the Priority Programme 1962 under project number \href{https://gepris.dfg.de/gepris/projekt/423457678}{423457678}.

The first and second author acknowledge funding of the project \textit{Fluiddynamische Formoptimierung mit Phasenfeldern und Lipschitz-Methoden} by the German Research Foundation under project number \href{https://gepris.dfg.de/gepris/projekt/543959359}{543959359}.

\subsubsection*{Author Contribution.}
(CRediT taxonomy)
\begin{description}
    \item [M.H.] Conceptualization, Formal analysis, Funding acquisition, Methodology, Project Administration, Resources, Supervision, Validation, Writing – Original Draft Preparation, Writing – review \& editing;
    \item [C.K.] Conceptualization, Formal analysis, Funding acquisition, Methodology, Project Administration, Resources, Supervision, Validation, Writing – Original Draft Preparation, Writing – review \& editing;
    \item [M.S.] Conceptualization, Data curation, Formal analysis, Investigation, Methodology, Software, Visualization,  Validation, Writing – Original Draft, Writing – review \& editing;
\end{description}

\printbibliography

\end{document}